\documentclass[a4paper,11pt]{article} 

\title{Large Complex Correlated Wishart Matrices:\\
The Pearcey Kernel and Expansion at the Hard Edge.} 

\author{
\ Walid Hachem
\footnote{
CNRS LTCI; T\'{e}l\'{e}com ParisTech, 46 rue Barrault, 75634 Paris Cedex 13, France. Email: walid.hachem@telecom-paristech.fr}\;\;,\;
\ Adrien Hardy  
\footnote{Department of Mathematics, KTH Royal Institute of Technology, Lindstedtsv\"agen 25, 10044 Stockholm, Sweden. Email: ahardy@kth.se}\;\;,\; 
\ Jamal Najim 
\footnote{CNRS LIGM; Universit\'e Paris-Est, 
Cit\'e Descartes, 
5 Boulevard Descartes, 
Champs sur Marne, 
77 454 Marne-la-Vall\'ee Cedex 2,
France. Email: najim@univ-mlv.fr}}

\usepackage[english]{babel}
\usepackage[utf8]{inputenc}

\usepackage{a4wide}
\usepackage{fancyhdr}
\usepackage{amssymb,amsmath,amscd,amsfonts,amsthm,bbm}
\usepackage{tikz,graphicx,subfigure,overpic}
\usepackage{color} 
\usepackage[colorlinks=true, linkcolor=blue, citecolor=blue]{hyperref}
\usepackage[fixlanguage]{babelbib}
\usepackage{color}
\usepackage{hyperref}
\usepackage{comment}

\numberwithin{equation}{section}

\def\Tr{\mathop{\mathrm{Tr}}\nolimits}
\def\re{\mathop{\mathrm{Re}}\nolimits}

\newtheorem{theorem}{Theorem}[]
\newtheorem{lemma}{Lemma}[section]
\newtheorem*{theorem*}{Theorem}
\newtheorem{corollary}[lemma]{Corollary}
\newtheorem{proposition}[theorem]{Proposition}
\theoremstyle{definition} 
\newtheorem{definition}[lemma]{Definition}
\newtheorem{assumption}{Assumption}[]
\newtheorem{Remark}[lemma]{Remark}
\newenvironment{remark}{\begin{Remark}\rm}{\end{Remark}}

\newtheorem{Example}[lemma]{Example}

\newcommand{\eq}{\begin{equation}}
\newcommand{\qe}{\end{equation}}

\newcommand{\R}{\mathbb{R}}
\newcommand{\C}{\mathbb{C}}

\newcommand{\p}{\mathbb{P}}
\newcommand{\K}{ {\rm K}}

\newcommand{\A}{{\rm A}}
\newcommand{\B}{{\rm B}}
\newcommand{\E}{{\rm E}}
\newcommand{\Q}{{\rm Q}}

\newcommand{\Up}{\Upsilon^+}

\newcommand{\Um}{\Upsilon^{-}}

\newcommand{\Ttilde}{\Xi}

\DeclareMathOperator{\supp}{Supp}
\DeclareMathOperator{\interior}{int}

\DeclareMathOperator{\sign}{sign}

\newcommand{\Be}{{\rm Be}} 
\newcommand{\Kbe}{\K_{\Be}^{(\alpha)}}
\newcommand{\Pe}{{\rm Pe}} 
\newcommand{\Kpe}{\K_{\Pe}^{(\tau)}}
\newcommand{\Kcheck}{\check \K_{\Pe}^{(\tau)}}

\newcommand{\bt}{\textbf}
\newcommand{\bs}{\boldsymbol}
\newcommand{\bv}{\mathbf}
 
\renewcommand{\leq}{\leqslant}
\renewcommand{\geq}{\geqslant}
\renewcommand{\le}{\leqslant}
\renewcommand{\ge}{\geqslant}
\renewcommand{\epsilon}{ \varepsilon}
\renewcommand{\d}{ {\rm d}}
\renewcommand{\frak}{\mathfrak} 
\renewcommand{\emptyset}{\varnothing}

\renewcommand{\Im}{\frak{Im}} 
\renewcommand{\re}{\frak{Re}}

\newcommand{\red}{\color{red}}
\newcommand{\blue}{\color{blue}}

\newcommand{\bigO}[1]{{\mathcal O} \left( {#1}\right)}

\begin{document}
\maketitle

\begin{abstract} 
We study the eigenvalue behaviour of large complex correlated Wishart matrices near an interior point of the limiting spectrum where the density vanishes (cusp point), and refine the existing results at the hard edge as well. More precisely, under mild assumptions for the population covariance matrix, we show that the limiting density  vanishes at generic cusp points like a cube root, and that the local eigenvalue behaviour is described by means of the Pearcey kernel if an extra decay assumption is satisfied. As for the hard edge, we show that the density blows up like an inverse square root at the origin. Moreover, we provide an explicit formula for the $1/N$ correction term for the fluctuation of the smallest random eigenvalue.


\end{abstract}


\setcounter{tocdepth}{2}

\setcounter{tocdepth}{2}
\tableofcontents

\section{Introduction}


Empirical covariance matrices are natural random matrix models in applied mathematics and their study goes back at least to the work of Wishart \cite{wishart-1928}.
In the large dimensional regime, where both the size of the observations and of the sample go to infinity at the same speed, Mar\v cenko and Pastur provided in the seminal paper  \cite{mar-and-pas-67} the first description of the limiting spectral distribution for such matrices, see also \cite{silverstein-choi-1995}. For instance, this limiting distribution has a continuous density on $(0,\infty)$; its support is compact if the spectral norm of the population covariance matrix is bounded; it may include the origin and may also present several connected components. 

Afterwards, attention turned to the local behaviour of the random eigenvalues near points of interest in the limiting spectrum, like positive endpoints (soft edges), see e.g. \cite{Jo,Joh01,BBP-2005,EK,HHN-preprint}, interior points were the density vanishes (cusp points) \cite{Mo2}, or  the origin when it belongs to the spectrum (hard edge) \cite{F,HHN-preprint}.  Complex correlated Wishart matrices, namely covariance matrices with complex Gaussian   entries, play a particular role in such investigations since their random eigenvalues form a determinantal point process. Indeed, for determinantal point processes, a local asymptotic analysis can often be performed by using  tools from complex analysis such as saddle point analysis or Riemann-Hilbert techniques. In the more general setting of non-necessarily Gaussian entries, one then typically shows that the local behaviours are the same as in the Gaussian case by comparison or interpolation methods, see e.g. \cite{knowles-yin-2014-preprint,LS14}.  


For complex correlated Wishart matrices, a fairly complete picture of the local fluctuations at every edges of the limiting spectrum has been obtained in the recent work \cite{HHN-preprint}, provided that a regularity condition is satisfied. This condition essentially warrants the local fluctuations to follow the usual  laws from random matrix theory.  For instance, if one considers a soft edge of the limiting spectrum, then this regularity condition ensures that the limiting density vanishes like a square root at the edge, and the fluctuations of the associated extremal eigenvalues follow the Tracy-Widom law involving the Airy kernel.   As for the hard edge, when it is present, the fluctuations are described instead by means of the Bessel kernel.



 The aim of this work is twofold. First, we investigate the local behaviour of the eigenvalues near a cusp point which satisfies the regularity condition: We show that the limiting density vanishes like a cube root near the cusp point (hence justifying the name) and, under an extra assumption on the decay of a speed parameter, we establish that the eigenvalues local fluctuations near the cusp point are described  by means of the Pearcey kernel. 
 
 
Our second contribution is to strengthen the results of \cite{HHN-preprint} concerning the local analysis at the hard edge: We show that the density behaves like an inverse square root near the origin, and we provide an explicit formula for the next-order correction term for the fluctuations. This last  result is motivated by the recent work \cite{EGP-preprint} by Edelman, Guionnet and P\'ech\'e where they conjecture a precise formula for the next-order term for the non-correlated Wishart matrix, a conjecture  then proven right by Bornemann \cite{bornemann-2014-note} and Perret and Schehr \cite{PS}, with different strategies. Our result  hence extends this formula, with an alternative proof, to the more general setting of  correlated Wishart matrices. 
 
The reader interested in a  pedagogical overview on the results from \cite{HHN-preprint} and the present work may have a look at the survey \cite{HHN-esaim-preprint};   it also contains further information on the matrix model and  lists some open problems.




Let us also stress that, at the technical level, the study of this matrix model shares similar features with the study of the additive perturbation of a GUE random matrix \cite{capitaine-peche-2014-preprint}, and random Gelfand-Tsetlin patterns \cite{DM1,DM2}, although each model ultimately brings up its own share of technicalities.

 We provide precise statements for our results in Section \ref{sec:Main results}, and then prove the results on the density behaviour in Section \ref{sec:density behaviours},  the cusp point fluctuations in Section \ref{sec:Pearcey}, and the expansion at the hard edge in Section \ref{sec:proof-expansion}.

\paragraph*{Acknowledgements.} 
The authors are pleased to thank Folkmar Bornemann,  Antti Knowles and Anthony Metcalfe for fruitful discussions.
During this work, AH was supported by the grant KAW 2010.0063 from the
Knut and Alice Wallenberg Foundation. 
The work of WH and JN was partially supported by the program 
``mod\`eles num\'eriques'' of the French Agence Nationale de la Recherche 
under the grant ANR-12-MONU-0003 (project DIONISOS). Support of Labex B\'EZOUT from Universit\'e Paris Est is also acknowledged. 

\section{Statement of the main results}
\label{sec:Main results}

\subsection{The matrix model and assumptions}
\label{subsec: matrix model}

The random matrix model of interest here is the $N\times N$ matrix 
\eq
\label{main matrix model}
{\bv M}_N = \frac 1N {\bf X}_N {\bf \Sigma}_N{\bf X}_N^* 
\qe
where  ${\bf X}_N$ is a $N\times n$ matrix with independent and identically
distributed (i.i.d.) entries with zero mean and unit variance, and 
${\bf \Sigma}_N$ is a $n\times n$ deterministic positive definite Hermitian 
matrix. The random matrix $\bv M_N$ thus has $N$ non-negative eigenvalues, but which may be
of different nature: The smallest $N-\min(n,N)$ eigenvalues are
deterministic and all equal to zero, whereas the other $\min(n,N)$ eigenvalues
are random. The problem is then to describe the asymptotic behaviour of the
random eigenvalues of $\bv M_N$,  as the size of the matrix grows to infinity. As for the asymptotic regime of interest, we let both the number of rows and columns of ${\bf X}_N$  grow to infinity at the same speed:  We assume $n=n(N)$ and $n,N\to \infty$  so that 
\eq
\label{evdistrMN}
\lim_{N\rightarrow\infty} \frac{n}{N}=\gamma\in (0,\infty)\ .
\qe 
This regime will be simply referred to as  $N\to \infty$ in the sequel.

Let us mention that the $n\times n$ random covariance matrix  
\[
\widetilde{\bv M}_N = \frac 1N 
{\bf \Sigma}_N^{1/2} {\bf X}_N^* {\bf X}_N {\bf \Sigma}_N^{1/2}\ ,
\]
which is also under consideration, has exactly
the same random eigenvalues as $\bv M_N$, and hence  results on the random
eigenvalues can be carried out from one model to the other immediately. 

Our first assumption is that the entries of $\bs X_N$ are complex Gaussian. As we shall state later on, this assumption is fundamental for our local eigenvalue behaviour analysis, but not for our results on the limiting density behaviour, see Remark \ref{exact assumption density}.


\begin{assumption} 
\label{ass:gauss} 
The entries of $\bv X_N$ are i.i.d. standard complex Gaussian random variables.
\end{assumption}

Considering now the matrix $\bv \Sigma_N$, we denote by $0<\lambda_1\leq \cdots\leq \lambda_n$ its eigenvalues and by 
\eq
\label{nuN}
\nu_N=\frac 1n \sum_{j=1}^n \delta_{\lambda_j} 
\qe  
its spectral measure. We also make the following assumption. 

\begin{assumption}\  
\label{ass:nu}
\begin{enumerate}
\item[(a)]
For $N$ large enough, the eigenvalues of $\bv \Sigma_N$ stay in a compact 
subset of $(0,+\infty)$ independent of $N$, i.e.
\eq
0\ <\ \liminf_{N\rightarrow\infty}\lambda_1\ ,\quad 
\sup_{N}\lambda_n\ <\ +\infty. 
\qe
\item[(b)]
The measure $\nu_N$ weakly converges towards a limiting probability measure
$\nu$ as $N\rightarrow\infty$, namely
\eq
\label{weak conv nuN}
\frac{1}{n}\sum_{j=1}^n f(\lambda_j)\xrightarrow[N\to\infty]{} \int f(x)\nu(\d x)
\qe
for every bounded and continuous function $f$.
\end{enumerate}
\end{assumption}
Again, Assumption \ref{ass:nu}(a) is necessary for our results on the local eigenvalue behaviour, but our results on the limiting density behaviour require a weaker assumption, see Remark \ref{exact assumption density}.

We now turn to the description of the asymptotic eigenvalue distribution. 

\subsection{Limiting eigenvalue distribution}
\label{sec:fixed-point}

Consider the empirical distribution of the eigenvalues $(x_i)$ of ${\bf M}_N$, namely
\[
\mu_N  =\frac 1N \sum_{i=1}^N \delta_{x_i}\, .
\] 
Since the seminal work of Mar\v cenko and Pastur~\cite{mar-and-pas-67}, it is known that this measure
 almost surely (a.s.) converges weakly towards a 
limiting probability measure $\mu$ with compact support, provided that Assumption \ref{ass:nu} holds true:
\eq
\label{mu}
\frac{1}{ N}\sum_{i=1}^N f(x_i)\xrightarrow[N\to\infty]{a.s.} 
\int f(x)\mu(\d x)
\qe
for every bounded and continuous function $f$. As a probability measure, $\mu$
is  characterized by its Cauchy-Stieltjes transform, which is the holomorphic function 
defined by 
\eq
\label{CS m}
m(z)=\int \frac{1}{z-\lambda}\,\mu(\d \lambda),\qquad 
z\in \C_+ =\big\{z\in\C : \; {\Im}(z)>0\big\}\, .
\qe
Mar\v cenko and Pastur proved that $m(z)$ is   the unique solution $m \in \C_-=\{z\in\C : \; {\Im}(z)<0\}$ of the fixed-point equation
 \begin{equation}
 \label{cauchy eq} 
m = \left( z - \gamma \int \frac{\lambda}{1 - m \lambda} \nu(\d\lambda) 
\right)^{-1}  , 
 \end{equation}
where we recall that $\gamma$ has been introduced in  \eqref{evdistrMN} and $\nu$ is the weak limit  of $\nu_N$, see \eqref{weak conv nuN}. 
Thanks to this equation, Silverstein and Choi then showed in~\cite{silverstein-choi-1995} that $\mu(\{0\}) = (1-\gamma)^+$ and  $\lim_{z\in\C_+ \to x} m(z) \equiv m(x)$ exists for every $x \in
\R^* = \R - \{ 0 \}$. Consequently, the function $m(z)$ can be continuously
extended to $\C_+ \cup \R^*$ and, furthermore, $\mu$ has a density on $(0,\infty)$
given by 
\eq
\label{rho}
\rho(x) = - \frac1\pi {\Im} \, (m(x))\, .
\qe 
We therefore have the 
representation 
\eq
\label{mu decomposition}
\mu(\d x)=\left(1-\gamma\right)^+ \, \delta_0 + \rho(x)\d x\ .
\qe
They also obtained that $\rho(x)$ is real analytic wherever it is positive, and
they moreover characterized the (compact) support $\supp(\rho)$ of the measure $\rho(x)\d x$ by building on  
 ideas from~\cite{mar-and-pas-67}. More specifically, one can see that the 
function $m(z)$ has an explicit inverse (for the composition law) on $m(\C_+)$ 
given  by
\begin{equation}
\label{g(m)} 
g(m) = \frac{1}{m} + 
    \gamma \int \frac{\lambda}{1 - m \lambda} \,\nu(\d \lambda)\ .
\end{equation}
If we introduce the open subset of the real line
\begin{equation}\label{def:D}
D=\bigl\{x\in\R : \ x\neq 0, \; x^{-1}\notin \supp(\nu)\bigr\} \ ,
\end{equation}
then the map $g$ analytically extends to $\C_+ \cup \C_- \cup D$. It was shown in \cite{silverstein-choi-1995} that 
\eq
\label{outside support}
\R - \supp(\rho) = \big\{ g(m) : \;  m\in D, \ g'(m) < 0 \big\} .
\qe

Equipped with the definitions of $m$, $g$ and $D$, we are now able to state our  results concerning the behaviours of the limiting density $\rho(x)$ near a cusp point or at the hard edge. 




\subsection{Density behaviour  near a cusp point} 

\begin{figure}[h]
\centering
\includegraphics[width=0.7\linewidth]{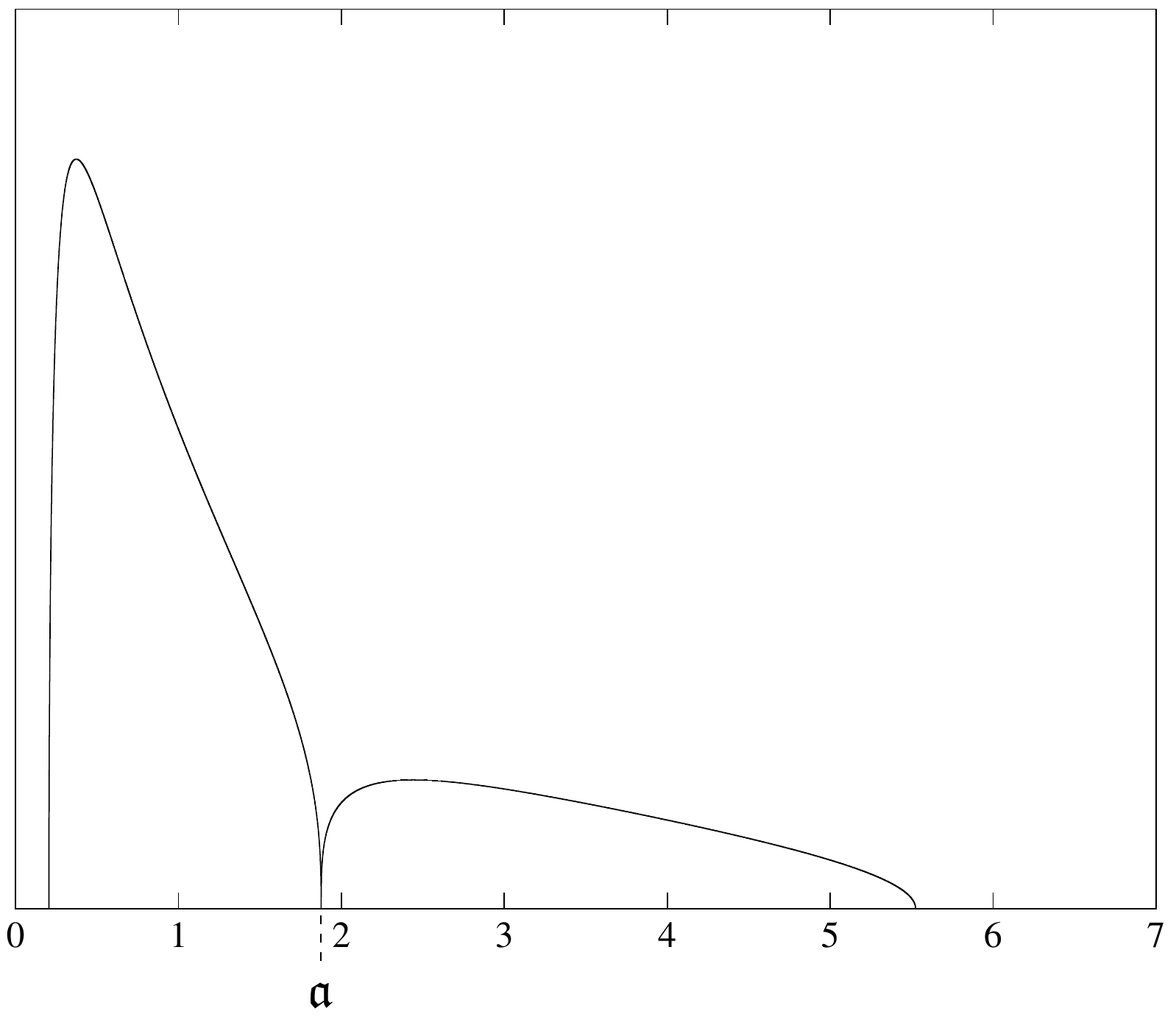}
\caption{Plot of  $\rho(x)$ with parameters $\gamma\simeq 0.336$ and 
$\nu = 0.7\delta_1 + 0.3\delta_3$ with a cusp point $\frak a$. }
\label{fig:cusp}
\end{figure}


As stated in the introduction, we define a \textbf{cusp point} $\frak a$ as an interior point where the density vanishes, namely 
$\frak a \in \interior(\supp(\rho))$ such that $\rho(\frak a) = 0$. In particular,  $\frak c=m(\frak a)\in\R$  by virtue of \eqref{rho}. Our first result states that the density $\rho(x)$ behaves like a cube root near a cusp point, provided that $\frak c\in D$.


\begin{proposition}
\label{cusp=>inflexion} 
Let $\frak a \in \interior(\supp(\rho))$ be such that $\rho(\frak a) = 0$, and assume that $\frak c = m(\frak a) \in D$. Then we have
$$
g(\frak c) = \frak a\, ,\qquad g'(\frak c) = g''(\frak c) = 0\, ,\qquad  \textrm{and} \quad g^{(3)}(\frak c) > 0\ .
$$
Moreover, 
\begin{equation} 
\label{cubic-root} 
\rho(x) = \frac{\sqrt{3}}{2\pi} 
\left(\frac{6}{g^{(3)}(\frak c)}\right)^{1/3}  \big|x-\frak a\big|^{1/3}(1+o(1)) \ ,\qquad x\rightarrow \frak a\, .
\end{equation} 
In particular, there exists $\eta>0$ such that for every $x\in (\frak a -\eta,\frak a +\eta)\setminus \{\frak a\}$, we have $\rho(x)>0$.
\end{proposition}


\begin{remark} In the forthcoming local analysis for the random eigenvalues near a cusp point, we shall focus on cusp points $\frak a$'s satisfying a regularity condition. This extra assumption automatically yields that $m(\frak a)\in D$, see Remark \ref{RC -> cubic}.
\end{remark}

Conversely, we have the following result. 
\begin{proposition}
\label{inflexion=>cusp} 
If $\frak c \in D$ satisfies $g'(\frak c) = g''(\frak c) = 0$, then
 $\frak a = g(\frak c)$ belongs to $\interior(\supp(\rho))$ and $\rho(\frak a) = 0$. 
In particular, $g^{(3)}(\frak c) > 0$ and $\rho(x)$ satisfies~\eqref{cubic-root}. 
\end{proposition}

 We prove Propositions \ref{cusp=>inflexion} and \ref{inflexion=>cusp} in sections \ref{proof:cusp=>inflexion} and \ref{proof:inflexion=>cusp} respectively.  Their proofs are based  on the fact that there is a strong relation between the property that $\frak a=g(\frak c)$ is a cusp point and the local behaviour of $g$ near $\frak c$. For an illustration of these propositions, we refer to Figure \ref{fig:gcusp} where we displayed the graph of the map $g$ associated with the density from Figure~\ref{fig:cusp}.

\begin{figure}[h]
\centering
\includegraphics[width=0.7\linewidth]{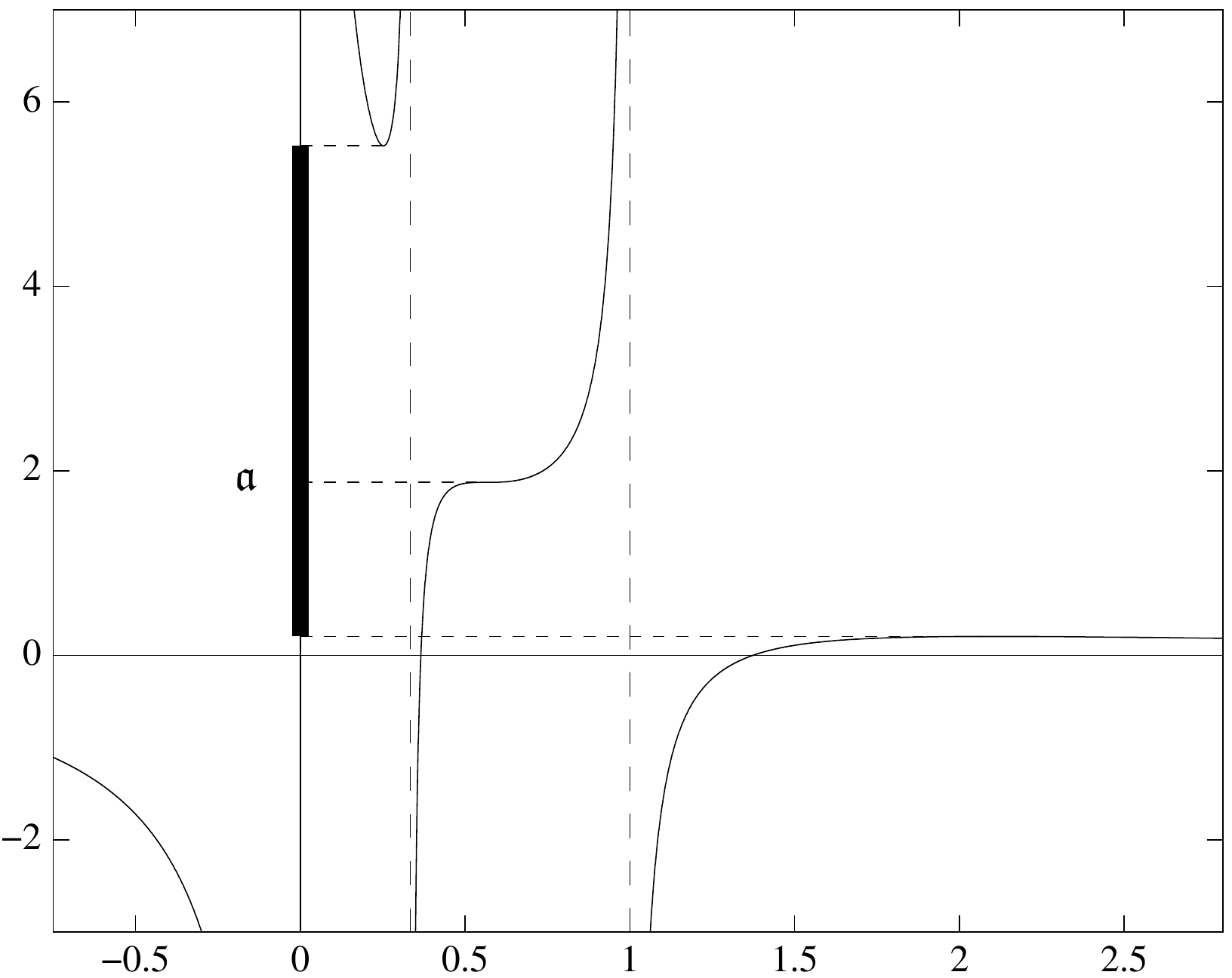}
\caption{Plot of function $x\mapsto g(x)$ on $D$ for $\gamma\simeq 0.336$ and 
$\nu = 0.7\delta_1 + 0.3\delta_3$. The vertical dotted lines are $g$'s asymptotes at $x=1/3$ and $x=1$. The thick segment on the vertical axis 
represents $\supp(\rho)$ who is delimited by the local extrema of g (see Eq. \eqref{outside support}). The point $\frak a$ is a cusp point. 
}
\label{fig:gcusp}
\end{figure}

We now turn to the hard edge setting. 

\subsection{Density behaviour near the hard edge}

As usual in random matrix theory, the \textbf{hard edge} refers here to the origin when it belong to the limiting spectrum. In general, the limiting eigenvalue distribution may not display a hard edge,  like in Figure \ref{fig:cusp}. In fact, this is always  the case when $\gamma\neq1$, see \cite[Proposition 2.4 (a),(c)]{HHN-preprint}. Our next result states that there is a hard edge, namely $0\in\supp(\rho)$, if and only if $\gamma=1$, and that in this case $\rho(x)$ blows up like an inverse square root at the origin. We furthermore relate the presence of a hard edge to the behaviour of $g$ near $\infty$. More precisely, since $\supp(\nu)\subset (0,\infty)$ by Assumption~\ref{ass:nu}, one can see from the definition \eqref{g(m)} of $g$  that the map $ g(1/z)$ is holomorphic at the origin. Thus, we have the analytic expansion as $z\to0$
\eq
\label{exp near infty}
g(1/z)= g(\infty) + g'(\infty) z + \frac{ g''(\infty)}{2} z^{2} + \cdots
\qe
Clearly, $g(\infty)=0$ and the coefficients $g'(\infty)$ and $g''(\infty)$ are respectively given by the first and second  derivative of the map $z\mapsto g(1/z)$ evaluated at $z=0$.



\begin{proposition}
\label{prop:hard-edge-density}
The  following three assertions are equivalent: 
\begin{itemize}
\item[\emph{i)}] $0\in \supp(\rho)$
\item[\emph{ii)}] $\gamma = 1$
\item[\emph{iii)}] $g'(\infty)=0$
\end{itemize}
Moreover, 
 we have $g''(\infty)=-2\gamma\int \lambda^{-1}\nu(\d \lambda) <0$ and, if  one of these assertions is satisfied, then 
\eq
\label{inv square root behav}
\rho(x) = \frac 1\pi \left( \frac 2{-g''(\infty)}\right)^{-1/2} x^{-1/2} \, (1+o(1))\ ,\qquad x\rightarrow 0_+\, .
\qe
\end{proposition}

Proposition \ref{prop:hard-edge-density} is proven in Section \ref{proof:hard-edge-density}.

\begin{remark}\label{analogy-soft} There is an analogous statement for any left edge $\frak a>0$ of the spectrum satisfying $\frak c=m(\frak a)\in D$ which follows from \cite{silverstein-choi-1995}; see also \cite[Section 2]{HHN-esaim-preprint} for further information.
 Indeed, in this case we have 
$$
g(\frak c) = \frak a,\qquad  g'(\frak c)=0, \qquad g''(\frak c)<0,
$$ 
and furthermore,  as $x\rightarrow \frak a_+$,
$$
\rho(x) = \frac 1\pi \left( \frac 2{-g''(\frak c)}\right)^{1/2} \big(x-\frak a\big)^{1/2}(1+o(1))\ .
$$
By analogy with this equation, the preimage $\frak c\in D$ corresponding to the hard edge is $\frak c=\infty$. The fact that it actually belongs to $D$ follows from  Assumption \ref{ass:nu}(a).
\end{remark}

\begin{remark} 
\label{exact assumption density}
As we shall see in Section \ref{density proofs}, proofs of Propositions  \ref{cusp=>inflexion}, \ref{inflexion=>cusp}  and \ref{prop:hard-edge-density} only rely on the properties of the limiting eigenvalue distribution $\mu$, which do not depend on whether the entries of $\bv X_N$ are Gaussian or not. More precisely, the exact assumptions required for these propositions are that the entries of $\bv X_N$ are  i.i.d centered random variables with variance one,  Assumption \ref{ass:nu}--(b), and that  $\mu(\{0\})=0$ (which follows from Assumption~\ref{ass:nu}--(a)), see \cite{mar-and-pas-67,silverstein-choi-1995}.
\end{remark}


\subsection{The Pearcey kernel and fluctuations near a cusp point}

Our next result essentially states that the random  eigenvalues of $\bv M_N$, properly scaled near a regular cusp point, asymptotically behave  like the determinantal point process associated with the Pearcey kernel, provided that an extra condition on a speed parameter is satisfied.  

In order to state this result, we first introduce this limiting point process. Next, we define what we mean by regular, and provide the existence of appropriate scaling parameters. After that, we  finally state our result for the fluctuations near a cusp point. 

\subsubsection{Determinantal point processes}
\label{sec:DPP}

A \textbf{point process} on $\R$ (or in a subset therein), namely a probability distribution $\p$ over the locally finite discrete subsets $(y_i)$ of $\R$, is \textbf{determinantal} if there exists an appropriate kernel $\K(x,y):\R\times\R\to\R$ which characterizes the correlation functions in the following way: For every $k\geq 1$ and every compactly supported Borel function $\Phi:\R^k\to\R$,  we have
\begin{align}
\label{cor function}
\mathbb E\left[   \sum_{y_{i_1}\neq \, \cdots \,\neq \, y_{i_k}}  \Phi(y_{i_1},\ldots,y_{i_k})\right]
=\int_\R \cdots\int_{\R} \Phi(y_1,\ldots,y_k)\det\Big[\K(y_i,y_j)\Big]_{i,j=1}^k\d y_1\cdots \d y_k\, .
\end{align}
In particular, the gap probabilities can be expressed as Fredholm determinants. Namely, given any interval $J\subset \R$, the probability that the point process avoids  $J$ reads
\eq
\label{gap proba}
\p\Big( (y_i)\cap J =\emptyset\Big)= 1+\sum_{k=1}^\infty\frac{(-1)^k}{k!}\int_J \cdots\int_{J}\det\Big[\K(y_i,y_j)\Big]_{i,j=1}^k\d y_1\cdots \d y_k\, ,
\qe
and the right hand side is the Fredholm determinant $\det(I-\K)_{L^2(J)}$ of the integral operator acting on $L^2(J)$ with kernel $\K(x,y)$, provided that it makes sense.  For instance, if one assumes that $J$ is a compact interval, which is enough for the purpose of this work, then $\det(I-\K)_{L^2(J)}$ is well-defined and finite as soon as ${\|\K\|}_J=\sup_{x,y\in J}|\K(x,y)|<\infty$. Moreover, the map $\K(x,y)\mapsto \det(I-\K)_{L^2(J)}$ is Lipschitz with respect to  ${\|\cdot\|}_J$ when restricted to the kernels satisfying  ${\|\K\|}_J<\infty$ (see e.g. \cite[Lemma 3.4.5]{AGZ}). We refer the reader to \cite{Hu,JoR,AGZ} for further information on determinantal point processes. 

\subsubsection{The Pearcey kernel}

Given any $\tau\in\R$, consider the Pearcey-like integral functions
\[
\phi(x)=\frac{1}{2i\pi}\int_{\Sigma} e^{xz-\tau \frac{z^2}2+\frac{z^4}4} \d z,\qquad \psi(y)=\frac{1}{2i\pi}\int_{-i\infty}^{i\infty} e^{-yw+\tau \frac {w^2}{2}-\frac{w^4}4} \d w,
\]
where the contour $\Sigma$ has two non-intersecting components, one which goes from $e^{i\pi/4}\infty$ to $e^{-i\pi/4}\infty$, whereas the other one goes from $e^{-3i\pi/4}\infty$ to $e^{3i\pi/4}\infty$. More precisely, we parametrize here this contour by 
 \eq
\label{Sigma}
\Sigma=\left\{\pm e^{i\theta}:\; \theta\in[-\pi/4,\pi/4]  \right\}\,\cup\,\left\{te^{\pm i\pi/4}:\;  t\in(-\infty,-1]\cup[1,+\infty)\right\}\, ,
\qe 
with the orientation as  shown in Figure \ref{fig:pkernel}.

\begin{figure}[h]
\centering
\includegraphics[width=0.4\linewidth]{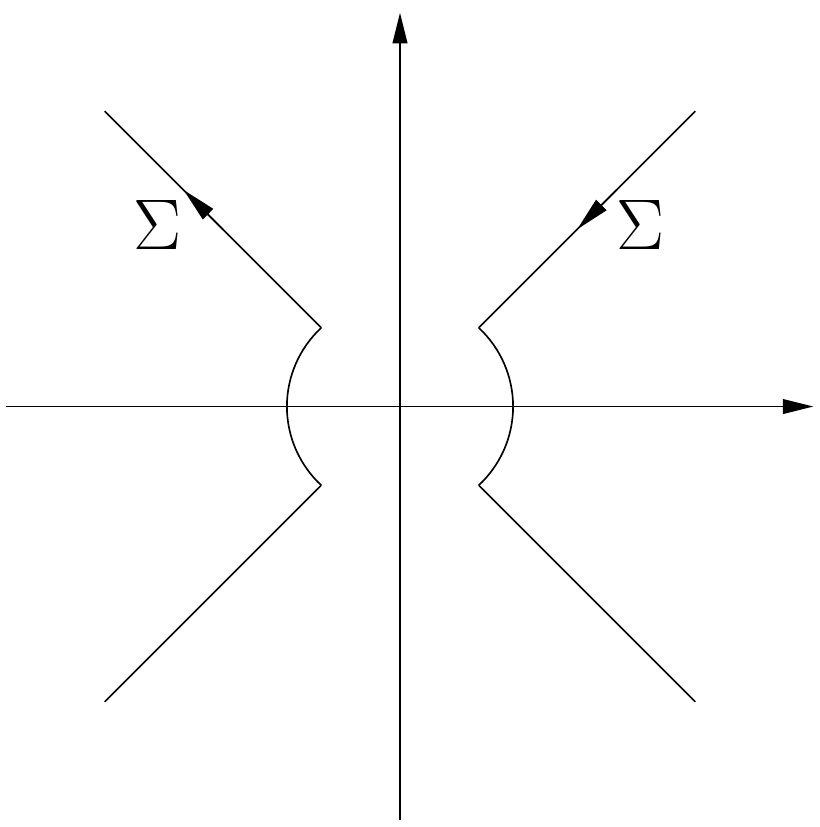}
\caption{The contour $\Sigma$.}
\label{fig:pkernel}
\end{figure}
It follows from their definitions that the functions $\phi$ and $\psi$ satisfy the respective differential equations 
\[
\phi'''(x)-\tau\phi'(x)+x\phi(x)=0\, , \qquad \psi'''(y)-\tau\psi'(y)-y\psi(y)=0\, .
\]
The \textbf{Pearcey kernel} is then defined for $x,y\in\R$ by
\eq
\label{KPe}
\K_\Pe^{(\tau)}(x,y)=\frac{\phi''(x)\psi(y)-\phi'(x)\psi'(y)+\phi(x)\psi''(y) -\tau\phi(x)\psi(y)}{x-y}\ ,
\qe
see for instance \cite{brezin-hikami-98-2}.
One can alternatively represent this kernel as a double contour integral
\eq
\label{Pearcey cont}
\Kpe(x,y)=
\frac{1}{(2i\pi)^2}\int_{\Sigma}\d z\int_{-i\infty}^{\, i\infty} \d w \,\frac{1}{w-z} e^{xz- \frac{\tau z^2}2+\frac{z^4}4-yw+ \frac{\tau w^2}2-\frac{w^4}4}\ ,
\qe
from which one can easily see the symmetry $\Kpe(x,y)=\Kpe(-x,-y)$ by performing the changes of variables $z\mapsto -z$ and $w\mapsto -w$.

The Pearcey kernel first appeared in the works of Br\'ezin and Hikami \cite{brezin-hikami-98,brezin-hikami-98-2} when studying the eigenvalues of a specific additive perturbation of a GUE random matrix near a cusp point. Subsequent generalizations have been considered by Tracy and Widom \cite{tracy-widom-pearcey}, and a Riemann-Hilbert analysis has been performed by Bleher and Kuijlaars  \cite{BK07} as well. This kernel also arises in the description of random combinatorial models, such as certain plane partitions \cite{OR07}. Furthermore, it has been established that the gap probabilities for the associated point process satisfy differential equations. For instance, $\log \det(I- \K^{(\tau)}_\Pe)_{L^2(s,t)}$ satisfies PDEs with respect to the variables $s$, $t$ and $\tau$, see \cite{tracy-widom-pearcey, BC12, ACvM12}, which should be compared to the connection between the Tracy-Widom distribution and the Painlev\'e II equation.

\subsubsection{The  regularity condition}

 We start with the following definition. 
 
\begin{definition}\label{def:regularity} A cusp point $\frak a$  is \textbf{regular} if $\frak c=m(\frak a)$ satisfies
\eq
\label{RC}
\liminf_{N\to\infty}\min_{j=1}^n\left|\frak c-\frac{1}{\lambda_j}\right|>0.
\qe
\end{definition} 

The regularity condition \eqref{RC} has  been considered in \cite{HHN-preprint} when dealing with soft edges, to ensure the appearance of the Tracy-Widom distribution. Similarly here, as we soon shall see, this condition enables the Pearcey kernel to arise at a cusp point. Moreover, the behaviour of $\rho(x)$ at such cusp points is well described by Proposition \ref{cusp=>inflexion}, as explained in the next remark. 

\begin{remark}
\label{RC -> cubic}
If $\frak a$ is a regular cusp point, then it follows from the weak convergence $\nu_N\to\nu$ and the definition of $D$ that necessarily $\frak c=m(\frak a)\in D$. Thus, $\frak a$  satisfies the hypothesis of Proposition \ref{cusp=>inflexion}. In particular,  we have $\frak a=g(\frak c)$ and $\rho(x)$ behaves like a cube root near $\frak a$.
\end{remark}

Finally, the regularity assumption yields the existence of natural scaling parameters for the eigenvalue local asymptotics. Consider the counterpart of the map $g$ introduced in \eqref{g(m)} after replacing $\nu$ by $\nu_N$, see \eqref{nuN}, and $\gamma$ by $n/N$, namely
\eq
\label{gN}
g_N(z)=\frac 1{z}+\frac nN\int\frac{\lambda}{1-z\lambda}\, \nu_N(\d \lambda) \, .
\qe
The map $g_N$ is the inverse Cauchy transform of a probability measure $\mu_N$ usually referred to as \textbf{deterministic equivalent} for the random eigenvalues distribution of $\bv M_N$, see~\cite[Section 3.2]{HHN-esaim-preprint}. According to Section \ref{sec:fixed-point}, $\mu_N$ has a decomposition of the form \eqref{mu decomposition}. In particular, it has a density $\rho_N$ on $(0,+\infty)$ which is analytic wherever it is positive.

The next proposition provides an appropriate  sequence of finite--$N$ approximations of $\frak c$, which  we will use in the definition of the scaling parameters.
\begin{proposition}
\label{prop:sequence-cN} Let $\frak a=g(\frak c)$ be a regular cusp point. Then there exists a sequence $(\frak c_N)$ of real numbers, unique up to a finite number of terms, converging to $\frak c$ and such that for every $N$ large enough, we have $\frak c_N\in D$ and 
$$
\lim_{N\to \infty} g_N(\frak c_N) = g(\frak c)\ ,\qquad \lim_{N\to \infty} g_N'(\frak c_N) = 0\ ,\qquad g_N''(\frak c_N) =0 \qquad \textrm{and} 
\qquad g_N^{(3)}(\frak c_N)>0\ . 
$$ 
\end{proposition}

This proposition is the counterpart of \cite[Proposition 2.7]{HHN-preprint}, with a similar proof. Let us only provide a sketch here: Combined with Montel's theorem, the regularity condition ensures  that $g_N^{(k)}$ converge uniformly to $g^{(k)}$ on a neighbourhood of $\frak c$ for every $k\geq 0$. The proposition  then follows by applying Hurwitz's theorem to $g''_N$ since $\frak c$ is a simple root for $g''$, according to Proposition \ref{cubic-root}. 

Let us emphasize that there is however an important difference with regular soft edges as described in \cite{HHN-preprint}, where it was shown that if $\frak a=g(\frak c)$ is a regular soft edge for $\mu$, then $g_N(\frak c_N)$ is a soft edge for $\mu_N$.

\begin{remark} A regular cusp point may not be the limit of finite--$N$ cusp points.
\label{rem:N-cusp}
More precisely, if $\frak a=g(\frak c)$ is a regular cusp point, then in particular $g'(\frak c)=g''(\frak c)=0$. However, this only ensures the existence of a sequence $(\frak c_N)$ such that 
$g_N''(\frak c_N)=0$. A priori, $g_N'(\frak c_N)$ converges to zero as $N\to \infty$ but might not be equal to zero.  In fact, it is not hard to show we have the following alternatives:
\begin{itemize}
\item if $g_N'(\frak c_N)=0$, then $g_N(\frak c_N)$ is a cusp point for $\mu_N$;
\item if $g'_N(\frak c_N)>0$, then the density $\rho_N$ is positive in a vicinity of $g_N(\frak c_N)$; 
\item if $g'_N(\frak c_N)<0$, then $g_N(\frak c_N)$ does not belong to the support of $\mu_N$.
\end{itemize}
\end{remark}

We are finally in position to state our result concerning the eigenvalue behaviour at a regular cusp point.


\subsubsection{Fluctuations around a cusp point}

Thanks to Assumption \ref{ass:gauss}, the random eigenvalues $(x_i)$ of $\bv M_N$ form a determinantal point process with respect to a kernel $\K_N(x,y)$, see \cite{BBP-2005,On}. An explicit formula for this kernel is provided in Section \ref{sec:Pearcey}. The main result of this section is the local uniform convergence of this kernel, properly scaled, towards the Pearcey kernel.

\begin{theorem}\label{th:main}
Let $\frak a=g(\frak c)$ be a regular cusp point. Let $(\frak c_N)$ be the sequence associated to $\frak c$ coming from Proposition~\ref{prop:sequence-cN}. Assume moreover that the following decay assumption holds true: There exists $\kappa\in \R$ such that
\begin{equation}\label{ass:speed}
\sqrt{N} g'_N(\frak c_N)\xrightarrow[N\to\infty]{} \kappa \ .
\end{equation}
Set
\begin{equation}\label{def:constants}
\frak a_N = g_N(\frak c_N)\ ,\qquad \sigma_N =\left( \frac{6}{g^{(3)}_N(\frak c_N)}\right)^{1/4}\ ,\qquad \tau = -\kappa \left( \frac 6{g^{(3)}(\frak c)} \right)^{1/2},
\end{equation}
so that $\frak a_N\to \frak a$ and $\frak \sigma_N \to (6/g^{(3)}(\frak c))^{1/4}>0$ as $N\to \infty$ by Proposition \ref{prop:sequence-cN}. Then, we have

\eq
\label{kernel conv}
\frac{1}{ N^{3/4} \sigma_N}\, \K_N\left(\frak a_N + \frac{x}{ N^{3/4} \sigma_N}, \frak a_N+\frac{y}{ N^{3/4} \sigma_N}\right)\xrightarrow[N\to\infty]{}\K_\Pe^{(\tau)}(x,y)
\qe
uniformly for $x,y$ in compact subsets of $\R$.
\end{theorem}

This result was obtained by Mo \cite{Mo2} in the special case where the matrix $\bv\Sigma_N$ has exactly two distinct eigenvalues (each with multiplicities proportional  $N$), by means of a Riemann-Hilbert asymptotic analysis.

Notice that if $(x_i)$ is determinantal with kernel $\K_N(x,y)$, then $(N^{3/4} \sigma_N (x_i - \frak a_N))$ is determinantal with the kernel given by the left hand side of \eqref{kernel conv}.  Thus, having in mind Section \ref{sec:DPP}, a direct consequence of this theorem is the convergence of the compact gap probabilities.

\begin{corollary} Under the assumptions of Theorem \ref{th:main}, we have for every  $s,t\in\R$,
\eq
\label{eq:hole}
 \mathbb{P}\left( 
\left( N^{3/4} \sigma_N (x_i - \frak a_N)\right) \cap [s,t] =\emptyset 
\right) 
\xrightarrow[N\to\infty]{}\det(I-\K^{(\tau)}_\Pe)_{L^2(s,t)} \ .
\qe
\end{corollary}

We now make a few comments on the assumption \eqref{ass:speed}.

\begin{remark} The decay assumption  \eqref{ass:speed} roughly states that the cusp point $\frak a=g(\frak c)$ appears fast enough. More precisely,  for a cusp point $\frak a=g(\frak c)$, one has in particular $g'(\frak c)=0$. 

\begin{itemize}
\item If $\kappa>0$, then $g'_N(\frak c_N)>0$ for large $N$. According to Remark \ref{rem:N-cusp}, the density $\rho_N$ is positive near $g_N(\frak c_N)$ and will converge to zero to asymptotically give birth to a cusp point. The family of densities $\rho_N$ display a sharp non-negative minimum at $g_N(\frak c_N)$ converging to zero which may be thought of as the  erosion of a valley, see the thin curve in Figure \ref{fig:zoomcusp}.
\item If $\kappa<0$, then $g'_N(\frak c_N)<0$ for large $N$ and the density $\rho_N$ vanishes in a vicinity of $g_N(\frak c_N)$. However, this interval will shrink and asymptotically disappear. Thus, two connected
components of the support of $\mu_N$ move towards one another (moving cliffs), see the dotted curve in
Figure \ref{fig:zoomcusp}. 
\end{itemize}
The assumption  \eqref{ass:speed} is an indication on the speed at which the bottom of the valley reaches zero $(\kappa>0)$, or at which the two cliffs approach one another $(\kappa<0)$. See \cite{HHN-esaim-preprint} for a more in-depth discussion.
\end{remark}


\begin{figure}[h]
\centering
\includegraphics[width=0.6\linewidth]{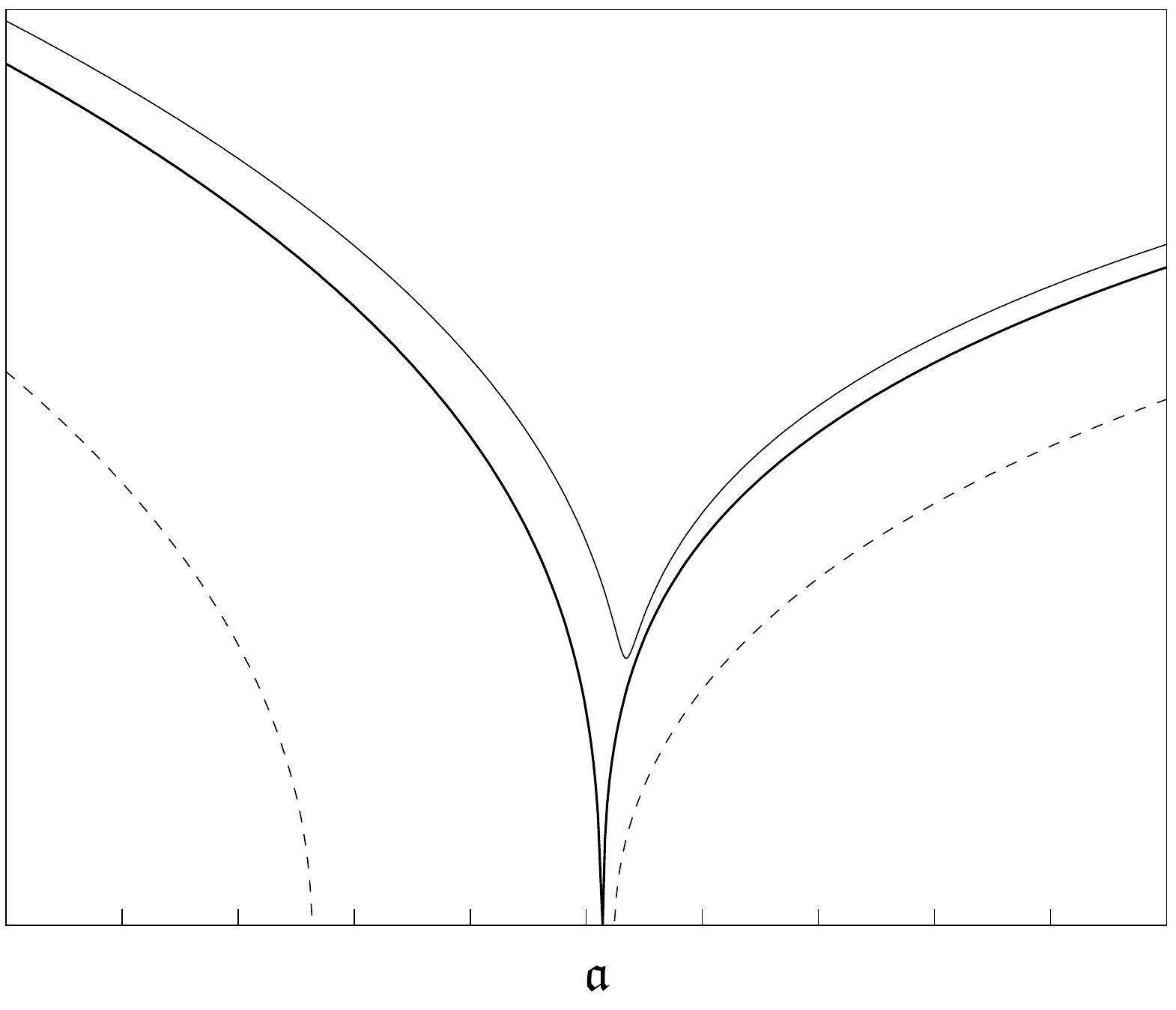}
\caption{Zoom of the density of $\mu_N$ near the cusp point
$\frak a$. The thick curve is the density of $\mu$ in the framework
of Figure~\ref{fig:cusp}. The thin curve (resp.~the dotted curve) is the 
density of $\mu_N$ when $\sqrt{N} g'_N(\frak c_N) > 0$
(resp.~$\sqrt{N} g'_N(\frak c_N) < 0$).}  
\label{fig:zoomcusp}
\end{figure}

\begin{remark}
When the assumption \eqref{ass:speed} is not satisfied, namely when $g'_N(\frak c_N)$ goes to zero  as $N\to\infty$ (which is always true by Proposition \ref{prop:sequence-cN}) slow enough so that $\sqrt{N} g'_N(\frak c_N)$ diverges to plus or minus infinity, we do not expect the Pearcey kernel to arise. See \cite{HHN-esaim-preprint} for further discussion.
\end{remark}


The proof of Theorem \ref{th:main} is provided in Section \ref{sec:proof-pearcey}.

\subsection{Asymptotic expansion at the hard edge}
\label{sec:description-BE} 

Our last result concerns the behaviour of the smallest random eigenvalue of $\bv M_N$ when the hard edge is present. Recall that $n/N\to\gamma$. By Proposition \ref{prop:hard-edge-density}, the limiting density displays a hard edge if and only if $\gamma=1$. With this respect, we restrict ourselves here to the case  $n=N+\alpha$ where $\alpha\in\mathbb Z$  is fixed and does not depend on $N$. 

\subsubsection{The Bessel kernel}

The Bessel function of the first kind $J_\alpha$ with parameter $\alpha\in\mathbb Z$ is defined by 
\eq
\label{series rep Bessel}
J_\alpha(x) = \left( \frac x2\right)^\alpha \sum_{n=0}^\infty \frac{(-1)^n}{n!\,  \Gamma(n+\alpha+1)} \left( \frac x2\right)^{2n},\qquad x>0\  ,
\qe
with the convention that, when $\alpha<0$, the first $|\alpha|$ terms in the series vanish (since the Gamma function $\Gamma$ has simple poles on the non-positive  integers).  

The \textbf{Bessel kernel} is then defined for $x,y>0$ by
\eq
\label{Bessel kernel}
\Kbe(x,y)=\frac{\sqrt{y}J_\alpha(\sqrt{x})J_\alpha'(\sqrt{y})-\sqrt{x}J_\alpha'(\sqrt{x})J_\alpha(\sqrt{y})}{2(x-y)}\ .
\qe
One can alternatively  express it as a double contour integral,
\eq
\label{Bessel kernel bis}
\Kbe(x,y)= \frac 1{(2i\pi)^2} \left( \frac yx\right)^{\frac \alpha 2} \oint_{|z|=r}\frac{\d z} z \oint_{|w|=R}\frac {\d w}w \frac 1{z-w} \left( \frac zw\right)^{\alpha}
e^{-\frac xz +\frac z4 +\frac yw -\frac w4}\ ,
\qe
where $0<r<R$ and the contours are simple and oriented counterclockwise, see  for instance  \cite[Lemma 6.2]{HHN-preprint}. We set for convenience 
\eq
F_{\alpha}(s) = \det\big(I-\Kbe\big)_{L^2(0,s)},\qquad s>0,
\qe
where the right hand side stands for the Fredholm determinant of the restriction to $L^2(0,s)$ of the integral operator $\Kbe$. According to \eqref{gap proba}, $F_{\alpha}(s)$ is the probability that the smallest particle of the determinantal point process associated with the Bessel kernel is larger than $s$. Tracy and Widom \cite{tracy-widom-bessel-94} established that certain simple transformations of $\log F_\alpha(s)$ satisfy Painlev\'e equations (Painlev\'e III and Painlev\'e V are involved).

\subsubsection{Correction for the smallest eigenvalue's fluctuations}
We denote by $x_{\min}$  the smallest random eigenvalue  of $\bv M_N$, namely
\eq
\label{x min}
x_{\min}=
\begin{cases}
x_1& \mbox{ if } \alpha\geq 0,\\
x_{1-\alpha}& \mbox{ if } \alpha<0.
\end{cases}
\qe
Our last result is stated as follows.


\begin{theorem} 
\label{th:edge-expansion} 
Assume $n=N+\alpha$ where $\alpha\in\mathbb Z$ is fixed and does not depend on $N$. Set
\begin{equation}\label{def:sigma-zeta}
\sigma_N=\frac{4}{N}\sum_{j={1}}^n\frac{1}{\lambda_j}\qquad \textrm{and}\qquad \zeta_N=\frac 8N \sum_{i=1}^N \frac 1{\lambda_i^2}\ ,
\end{equation}
so that 
$$
\sigma_N\xrightarrow[N\to\infty]{}  4 \int \frac {\nu(\d \lambda)}{\lambda}\ >\ 0\qquad \textrm{and}\qquad \zeta_N\xrightarrow[N\to\infty]{} 8 \int\frac {\nu(\d \lambda)}
{\lambda^{2}}\ >\ 0
$$ 
by Assumption \ref{ass:nu}. Then, for every $s>0$, we have as $N\to\infty$,
\eq
\p\Big(N^2\sigma_N\,x_{\min}\geq s\Big)=F_\alpha(s)  - \frac1N\left( \frac{\alpha \zeta_N }{\sigma_N^2}\right)  s\frac{\d }{\d s}F_{\alpha}(s) + \bigO{\frac 1{N^2}}\ .
\qe 
\end{theorem}

The convergence towards $F_\alpha(s)$ has been first observed by Forrester \cite{F} when $\bv\Sigma_N$ is the identity. As for the general $\bv \Sigma_N$ case, it has been established by the authors in \cite{HHN-preprint}. An explicit formula for the $1/N$-correction term was  conjectured when $\bv \Sigma_N$ is the identity by Edelman, Guionnet and P\'ech\'e \cite{EGP-preprint}, a conjecture proved true soon after by Perret and Schehr \cite{PS} and Bornemann \cite{bornemann-2014-note}, with different techniques.  We thus generalize this formula to the general $\bv\Sigma_N$ case. The strategy of the proof is rather similar to Bornemann's one: It relies on an identity involving the resolvent of $\K_\Be^{(\alpha)}$ obtained by Tracy and Widom, although we cannot rely on existing estimates for the kernel in this general setting.  

\begin{remark}
In fact, as we shall see in the proof of Theorem \ref{th:edge-expansion} (see Remark \ref{next order exp}), our method easily yields for every $s>0$ an expansion  of the form 
\eq
\label{hard edge exp ?}
\p\Big(N^2\sigma_N\,x_{\min}\geq s\Big) = F_\alpha(s) + \sum_{\ell=1}^L \frac{1}{N^\ell} C_{N,\ell}^{(\alpha)}(s) + \bigO{\frac{1}{N^{L+1}}}
\qe
for every $L\geq 1$ as $N\to\infty$. Although we are able to provide a close formula for the coefficient $C_{N,1}^{(\alpha)}(s)$ (as stated in Theorem \ref{th:edge-expansion}) thanks to a formula due to Tracy and Widom, to the best of our knowledge the next order coefficients do not seem to benefit from such a simple representation.  
\end{remark}

We prove Theorem \ref{th:edge-expansion} in Section \ref{sec:proof-expansion}.


 \section{Proofs of the limiting density behaviours}
 \label{density proofs}
This section is devoted to the proofs of Proposition \ref{cusp=>inflexion}, \ref{inflexion=>cusp} and \ref{prop:hard-edge-density}.

We first recall a few facts stated in Section \ref{sec:fixed-point} that we shall use in the forthcoming proofs: The map $m$ is the Cauchy-Stieltjes transform \eqref{CS m}; it is analytic on $\C_+$ and extends continuously  to $\C_+\cup\R^*$. Moreover, $\Im (m(x))=-\pi\rho(x)$ for every $x\in \R^*$. The map $g$ defined in \eqref{g(m)} is analytic on $\C_+\cup\C_-\cup D$. In particular, $g'$ has isolated zeroes on $D$.  Moreover, after noticing that $m(\C_+)\subset\C_-$, we have the identity $g(m(z))=z$ for every $z\in\C_+$.
 
We start with a simple but useful fact, which follows by taking the limit $z\in\C_+\to x $ in the previous identity and using the continuity of $m$ and $g$ on their respective domains:
 
 \begin{lemma}
 \label{simple obs}
 If $x\in\C_+\cup\R^*$ is such that $m(x)\in \C_-\cup D$, then  $g(m(x))=x$.
 \end{lemma}

 We will also use the following property.
 
 \begin{lemma}
 \label{g' vanishes}
 If $\frak a\in\supp(\rho)$ satisfies $\rho(\frak a)=0$ and $\frak c=m(\frak a)\in D$, then $g'(\frak c)=0$.
 \end{lemma}
 
 \begin{proof}
Consider  the map
\eq
\label{G map}
G(m)=-\frac{1}{|m|^2}+\gamma\int\frac{\lambda^2}{|1-\lambda m|^2}\nu(\d \lambda)\ ,\qquad m\in\C_-\cup D,
\qe
and notice it is continuous on its definition domain  by dominated convergence. It follows from the definition \eqref{g(m)} of $g$ that $G(m)=g'(m)$ when $m\in D$, and  moreover that for every $m\in\C_-$ we have the identity
\[
\Im(g(m))=\Im(m)  G(m).
\]
By using the fact that $g(m(z))=z$ on $\C_+$,  we thus obtain for every $z\in \C_+$,
\eq
\label{eq imag}
\Im (z) = \Im (m(z)) G(m(z)).
\qe
If $x\in \R^*$ is such that $\rho(x)>0$, then by letting $z\in\C_+\to x$ in \eqref{eq imag}  we see that necessarily $G(m(x))=0$ because  $\lim_{z\to x} \Im (m(z)) =\Im (m(x))=-\pi \rho(x)< 0$. Now, since $\frak a\in\supp(\rho)$ by assumption, there exists a sequence $(x_k)$ such that $x_k\to \frak a$ as $k\to\infty$ and $\rho(x_k)>0$, and hence $G(m(x_k))=0$. Since $m(x_k)\to m(\frak a)=\frak c$ and $\frak c\in D$, this yields $g'(\frak c)=G(\frak c)=0$. 
 \end{proof}

 \label{sec:density behaviours}

 \subsection{Density behaviour near a cusp point: Proof of Proposition \ref{cusp=>inflexion}}\label{proof:cusp=>inflexion}

 We now turn to the proof of the first proposition.
 
 \begin{proof}
 
  Assume that $\frak a\in\interior(\supp(\rho))$ and $\rho(\frak a)=0$.  Set $\frak c=m(\frak a)$ and assume moreover that $\frak c\in D$. Thus, the facts that $g(\frak c)=\frak a$ and $g'(\frak c)=0$ directly follows from Lemma \ref{simple obs} and Lemma \ref{g' vanishes}. 

First, we prove that $g''(\frak c)=0$. To do so, we show that $g'>0$ on $(\frak c-\eta,\frak c+\eta) - \{\frak c\}$ for some $\eta>0$. Since $g'(\frak c)=0$ this would indeed yield that $\frak c$ is a local extremum for $g'$. We proceed by contradiction: Assume there exists a sequence $(x_k)$ in $D-\{\frak c\}$ such that $x_k\to\frak c$ and $g'(x_k)\leq 0$.  Since $g'$ has isolated zeroes on $D$, necessarily $g'(x_k)<0$ for every $k$ large enough. It then follows from \eqref{outside support} that $g(x_k)\in \R- \supp(\rho)$  and, since $g(x_k)\to g(\frak c)=\frak a$, this contradicts the assumption that $\frak a\in\interior(\supp(\rho))$.

Next, we similarly show that there exists $\eta>0$ such that  
\begin{equation}\label{isolated}
 x\in (\frak a-\eta, \frak a+\eta)- \{\frak a\}\quad  \Rightarrow  \quad \rho(x)>0\ .
\end{equation} 
We will use \eqref{isolated} later on in this proof. Assume there exists a sequence $(x_k)$ in $\R^*-\{\frak a\}$ such that $x_k\to \frak a$ and $\rho(x_k)=0$. Since $\frak a\in\interior(\supp(\rho))$ and $m(\frak a)\in D$ by assumption, then $x_k\in\supp(\rho)$ and $m(x_k)\in D$ for every $k$ large enough. Moreover, we have $m(x_k)\to m(\frak a)$, but since Lemma \ref{g' vanishes} then yields $g'(m(x_k))=0$, this contradicts that $g'$ has isolated zeroes on $D$ and  \eqref{isolated} follows. 

Now, we show that $g^{(3)}(\frak c) > 0$ by direct computation. Recalling the definition \eqref{g(m)} of $g$, the equation  $g''(\frak c)=0$ reads
\[
- \frac{1}{\frak c^3}=\gamma \int \frac{\lambda^3}{(1-\lambda \frak c)^3}\, \nu(\d \lambda)\ .
\]
As a consequence, we obtain 
\begin{align*}
\frac16g^{(3)}(\frak c) & = - \frac 1{{\frak c}^4} +\gamma \int \frac{\lambda^4}{(1-\lambda{\frak c})^4} \,\nu(\d\lambda)\\
& =  \frac 1{{\frak c}}\left(\gamma \int \frac{\lambda^3}{(1-\lambda \frak c)^3}\,\nu(\d \lambda)\right) +\gamma \int \frac{\lambda^4}{(1-{\lambda\frak c})^4}\, \nu(\d\lambda) \\
& = \gamma \int \frac{\lambda^3}{\frak c (1-\lambda \frak c)^4}\,\nu(\d \lambda)>0\ .
\end{align*}

We finally turn to the proof of the cube root behaviour \eqref{cubic-root}.  
Since $g'(\frak c) = g''(\frak c) = 0$ and $g^{(3)}(\frak c) > 0$, there exists an analytic map $\varphi$ defined on a complex neighbourhood of $\frak c$ such that $g(m) - \frak a = \varphi(m)^3$ and $\varphi'(\frak c)\neq 0$, see e.g.  \cite[Theorem 10.32]{book-rudin-real-complex}.  In particular, we have $\varphi'(\frak c)^3 = g^{(3)}(\frak c) / 6\stackrel{\triangle}= C>0$. Moreover,  the inverse function theorem yields that $\varphi$ has a local inverse $\varphi^{-1}$, defined on a neighbourhood of zero, such that $\varphi^{-1}(0)=\frak c$ and $\left(\varphi^{-1}\right)^{'}(0)=  1/\varphi'(\frak c)$.  If $|x-\frak a|$ is small enough, then $\rho(x)>0$ by \eqref{isolated}, hence $m(x)\in \C_-$, and we have $g(m(x))=x$ by Lemma \ref{simple obs}.  In particular, we have 
 \[
\varphi(m(x))^3 = g(m(x)) - \frak a  =  x - \frak a 
\] 
and, by taking the cube root (principal determination), applying $\varphi^{-1}$, and performing a  Taylor expansion, we obtain
\eq
\label{m(x) exp cusp}
m(x) =\frak c + \boldsymbol{j}  \left(\frac{x-\frak a}{C}\right)^{1/3} + o(x-\frak a)^{1/3}
\qe
where $\bs j$ is an undetermined cube root of unity. Since $ m(x)\in\C_-$, necessarily $\boldsymbol{j}=\exp(2i\pi/3)$ if $x<\frak a$ and $\boldsymbol{j}=\exp(-2i\pi/3)$ if $x>\frak a$. Finally, \eqref{cubic-root} follows by taking the imaginary part in \eqref{m(x) exp cusp}.

Proof of Proposition \ref{cusp=>inflexion} is therefore  complete.

\end{proof}

 \subsection{Identification of a cusp point: Proof of Proposition \ref{inflexion=>cusp}} \label{proof:inflexion=>cusp}

The main part of the proof consists in showing the following lemma.

\begin{lemma}  
\label{reverse bij} 
Let $\frak c\in D$ such that $g'(\frak c)=0$. 
Then $g(\frak c)\neq 0$ and $m(g(\frak c))=\frak c$. 
\end{lemma}

Equipped with Lemma \ref{reverse bij}, let us first show how  the proposition follows.

\begin{proof}[Proof of Proposition \ref{inflexion=>cusp}] 

Assume that $\frak c\in D$ satisfies $g'(\frak c)=g''(\frak c)=0$ and set
$\frak a=g(\frak c)$.  In particular, $\frak a \in\supp(\rho)$ by
\eqref{outside support}, and we just have to show that $\rho(\frak a)=0$ and
$\frak a \notin \partial \supp(\rho)$. We know from Lemma~\ref{reverse bij} 
that $\frak a\neq 0$ and $m(\frak a)=\frak c$. 
As a consequence, $\rho(\frak a)=-\Im(m(\frak a))/\pi=0$. Finally, since $\frak a=g(\frak c)$ with $\frak c\in D$ and $\rho(\frak a)=0$, then \cite[Theorem 5.2]{silverstein-choi-1995} shows the condition $\frak a\in \partial \supp(\rho)$ requires $g''(\frak c)\neq 0$, which is not possible by assumption. 


\end{proof}

We now prove the lemma.

\begin{proof}[Proof of Lemma \ref{reverse bij}]
For every $m\in D$ we have 
\[
g(m)+m g'(m)=\gamma \int \frac{\lambda}{(1-\lambda m)^2}\,\nu (\d \lambda)
 \neq 0\ ,
\]
and because $g'(\frak c)=0$ by assumption, by taking $m=\frak c$ we see that 
necessarily $g(\frak c) \neq 0$. Let us now prove that 
$m(g(\frak c)) = \frak c$.
Introduce for convenience the map 
$$
\Phi(z)=\frac z{1-z}\ .
$$ 
By combining  the fixed point equation \eqref{cauchy eq} for $m(z)$ with $z\in\C_+$ and that 
\[
\frac{1}{\frak c} = g(\frak c)  - g(\frak c) + \frac{1}{\frak c} = g(\frak c) - \gamma \int \frac{\lambda}{1-\frak c \lambda}\,\nu(\d\lambda)\ ,
\]
we obtain
\begin{align*}
m(z)-\frak c & = \frak c m(z)\left( \frac1{\frak c} - \frac1{m(z)} \right)\\
& = \frak c m(z)( g(\frak c)-z) +\gamma (m(z)-\frak c)\int \Phi(\lambda m(z))\Phi(\lambda \frak c)\,\nu(\d \lambda)\ .
\end{align*}
By reorganizing this equation as 
\begin{equation}\label{eq:contrainte}
(m(z) - \frak c) 
\Bigl( 1 - \gamma 
\int \Phi(\lambda m(z)) \, \Phi(\lambda \frak c) \, \nu(\d\lambda) \Bigr) = 
 (g(\frak c)-z) m(z) \frak c \ ,
\end{equation}
we see the lemma would follow by taking the limit  $z\in \C_+ \to g(\frak c)\in\R^*$ assuming  that 
\eq
\label{LB to prove}
\left|1 - \gamma \int \Phi(\lambda m(z)) \, \Phi(\lambda \frak c) \, \nu(\d\lambda) \right| \geq C | m(z) - \frak c |^2 \ ,\qquad  z\in\C_+\to g(\frak c),
\qe
for some constant $C>0$. To show \eqref{LB to prove}, we start by writing  
\begin{align}
\label{ineq start}
& \left|1 - \gamma \int \Phi(\lambda m(z)) \, \Phi(\lambda \frak c) \, \nu(\d\lambda) \right| \nonumber \\
 &   \geq 1-\gamma\int \big|\Phi(\lambda m(z)) \, \Phi(\lambda \frak c)\big |\, \nu(\d\lambda)  \\
 & =  1 - \frac{\gamma}{2}\left\{\int \big|\Phi(\lambda m(z)) \big|^2\,  \nu(\d\lambda)\nonumber
 + 
\int \big| \Phi(\lambda \frak c) \big|^2\, \nu(\d\lambda)
-
\int \big|\Phi(\lambda m(z)) -\Phi(\lambda \frak c) \big|^2\, \nu(\d\lambda)
 \right\}\ .
\end{align}
Recalling the definition \eqref{G map} of $G$, we see from \eqref{eq imag}  that $G(m(z)) < 0$ on $\C_+$, and this yields 
 
\begin{equation}
\limsup_{z\in\C_+ \to g(\frak c)} \gamma 
\int \big| \Phi(\lambda m(z)) \big|^2 \, \nu(\d\lambda) \leq 1 \ .
\end{equation}
Similarly, the equation $G(\frak c)=g'(\frak c) =0$ reads 
\eq
\gamma \int \big|\Phi(\lambda \frak c)\big|^2 \, 
\nu(\d\lambda) = 1 \ . 
\qe
Finally, we have 
\begin{align}
\int \big| \Phi(\lambda m(z)) - \Phi(\lambda \frak c) \big|^2
\nu(\d\lambda)  & = 
| m(z) - \frak c |^2 
\int \Bigl| \frac{\lambda}{(1-\lambda m(z))(1-\lambda \frak c)} \Bigr|^2
\nu(\d\lambda) \, \nonumber \\
& \geq 
| m(z) - \frak c |^2 
\int \Bigl| \frac{\lambda}{(1+\lambda |m(z)|)(1+\lambda |\frak c|)} \Bigr|^2
\nu(\d\lambda) 
\end{align}
and moreover,  since $m$ is continuous on $\C_+\cup\R^*$ and $g(\frak c)\in\R^*$, 
\eq
\label{ineq end}
\liminf_{z\in\C_+\to g(\frak c)}\int \Bigl| \frac{\lambda}{(1+\lambda |m(z)|)(1+\lambda |\frak c|)} \Bigr|^2
\nu(\d\lambda)  >0\ .
\qe 
The lower bound \eqref{LB to prove} then follows by combining \eqref{ineq start}--\eqref{ineq end} and the proof  is complete.
\end{proof}

\subsection{Density behaviour near the hard edge: Proof of Proposition \ref{prop:hard-edge-density}}
\label{proof:hard-edge-density}

The key to study the hard edge is to study the map $g$ near infinity, which is holomorphic there. More precisely, we have the analytic  expansion as $z\to 0$, 
\begin{align}
\label{exp g hard edge}
g(1/z)& = z -\gamma z \int \frac1{1-z/\lambda}\nu(\d \lambda) \nonumber\\
 &   = (1-\gamma) z -\left(\gamma\int \lambda^{-1}\nu(\d\lambda)\right) z^2 -\left(\gamma\int \lambda^{-2}\nu(\d\lambda)\right) z^3 + \cdots
\end{align}

The reason why one should do so is that $|m(x)|$ goes to $\infty$ as $x\searrow 0$:

\begin{lemma}
\label{lemma hard edge}
Assume $0\in\supp(\rho)$. Then, there exists $\eta>0$ such that $\rho(x)>0$ on $(0,\eta)$. Moreover, as $x>0$ decreases to zero, we have $|m(x)|\to+\infty$.
\end{lemma}

Equipped with this lemma, we first provide a proof of the proposition.

\begin{proof}[Proof of Proposition \ref{prop:hard-edge-density}]
A comparison between \eqref{exp near infty} and \eqref{exp g hard edge} readily yields $g'(\infty)=1-\gamma$ and $g''(\infty)=-2\gamma\int \lambda^{-1}\nu(\d \lambda)<0$. In particular, the equivalence between $ii)$ and $iii)$ is obvious. That $i) \Rightarrow ii)$ follows from \cite[Proposition~2.4(a)(c)]{HHN-preprint}. We show $ii)\Rightarrow i)$ by contradiction: Assume $\gamma=1$ and that there exists $0<a_0<\min\supp(\rho)$. Since $\gamma=1$, by \eqref{mu decomposition} we have $\mu(\d x)=\rho(x)\d x$ and thus $c_0=m(a_0)=\int (a_0- x)^{-1}\rho(x)\d x <0$. A close look at the definition~\eqref{g(m)} of $g$ shows that $g< 0$ on $(-\infty,0)$ (see also \cite[Proposition~2.4(b)]{HHN-preprint}), and thus $g(c_0)<0$. On the other hand, $c_0\in D$ since $c_0<0$ and hence Lemma \ref{simple obs} then yields $0<a_0=g(c_0)$, which is a contradiction.

We now prove the inverse square root behaviour \eqref{inv square root behav}.  Since $g'(\infty)  = 0$ and $g''(\infty) < 0$, there exists an analytic map $\varphi$ defined on a complex neighbourhood of the origin such that $g(1/z) = \varphi(z)^2$ and $\varphi'(0)\neq 0$, see   \cite[Theorem 10.32]{book-rudin-real-complex}.  In particular, we have $\varphi'(0)^2 = g''(\infty) / 2\stackrel{\triangle}= -C<0$. The map $\varphi$ is uniquely defined up to a sign, and thus $\varphi'(0)= \pm i C^{1/2}$. Moreover,  $\varphi$ has a local inverse $\varphi^{-1}$,  defined on a neighbourhood of zero,  so that $\varphi^{-1}(0)=0$ and $\left(\varphi^{-1}\right)^{'}(0)=  1/\varphi'(0)$.
 For every $x>0$ small enough, Lemma \ref{lemma hard edge} yields that $m(x)\in \C_-$, and thus $g(m(x))=x$ by Lemma \ref{simple obs}, and moreover that $1/m(x)$ lies in the definition domain of $\varphi$. As a consequence, 
 \[
\varphi(1/{m(x)})^2 = g(m(x))   =  x 
\] 
and, by taking the square root (principal determination), applying $\varphi^{-1}$, and performing a  Taylor expansion,  we obtain
\eq
\label{m(x) exp hard}
\frac1{m(x)} = \pm\frac 1{i}  \left(\frac{x}{C}\right)^{1/2} + o(x^{1/2})\ ,\qquad x\to 0_+\ .
\qe
Since $ m(x)\in\C_-$, the undetermined sign has to be a minus sign, and \eqref{inv square root behav} follows by taking the inverse and then the imaginary part in \eqref{m(x) exp hard}.

\end{proof}

We finally turn to the proof of the lemma.

\begin{proof}[Proof of Lemma \ref{lemma hard edge}] Let $z=x+iy$ with $x,y>0$ and assume that $\rho(x)>0$. The fixed point  equation \eqref{cauchy eq} yields
\begin{align*}
x\ & =\ \re\left(\frac1{m(z)} + \gamma\int \frac{\lambda}{1-\lambda m(z)}\, \nu(\d \lambda)\right)\\
 & = \ -\re (m(z)) G(m(z)) + \gamma \int \frac{\lambda}{|1-\lambda m(z)|^2} \,\nu(\d \lambda)\, ,
\end{align*}
where $G$ is defined as in \eqref{G map}. When $y\to 0$, since $\rho(x)>0$, we have $m(z)\to m(x)\in\C_-$ and  moreover $G(m(z))\to 0$ for the same reason as in proof of 
Lemma \ref{g' vanishes}. Thus, we obtain,
\eq
\label{eq x hard edge}
x  = \int \frac{\lambda}{|1-\lambda m(x)|^2} \,\nu(\d \lambda)\ .
\qe

We now show that $\rho(x)>0$ for every $x\in(0,\eta)$ for some $\eta>0$ by contradiction: Assume there exists a sequence $(x_k)$ such that $x_k>0$, $x_k\to 0$ as $k\to\infty$, and $\rho(x_k)=0$.  Since $0\in\supp(\rho)$, without loss of generality one can assume that $x_k\in\supp(\rho)$. As a consequence,  using moreover that $m$ is continuous on $\R^*$, one can construct a sequence $(y_k)$ satisfying $y_k>0$, $y_k\to 0$ as $k\to \infty$, $\rho(x_k)>0$, and $|m(x_k)-m(y_k)|\leq 1$. Because Assumption \ref{ass:nu} yields $\min\supp(\nu)>0$, we see from \eqref{eq x hard edge} that necessarily $|m(y_k)|\to+\infty$, and hence $|m(x_k)|\to+\infty$, as $k\to\infty$. In particular, $m(x_k)\in  D$ for $k$ large enough and Lemma \ref{g' vanishes} then yields $g'(m(x_k))=0$. Thus, $(1/m(x_k))$ is a vanishing sequence  as $k\to\infty$ of zeroes of the derivative of $g(1/z)$. But since $g(1/z)$ is holomorphic near zero, and hence so does its derivative, this contradicts the isolated zero principle and  our claim follows. 

Finally, since the identity \eqref{eq x hard edge} thus holds true on $(0,\eta)$,  that $|m(x)|\to+\infty$ follows by letting $x>0$ decrease to zero in this identity.
\end{proof}

\section{Fluctuations around a cusp point: Proof of Theorem \ref{th:main}}\label{sec:proof-pearcey}
\label{sec:Pearcey}

In this Section, we study the local behaviour  of the random eigenvalues  $(x_i)$ of $\bv M_N$ near a regular cusp point and prove Theorem \ref{th:main}. Our asymptotic analysis is based on that, thanks to Assumption \ref{ass:gauss}, the random eigenvalues form a determinantal point process with an explicit kernel $\K_N(x,y)$, see \cite{BBP-2005,On}. The kernel has the following double contour integral formula
\eq
\label{KN}
\K_N(x,y)=\frac{ N}{(2i\pi)^2}\oint_{\Gamma}\d z\oint_{\Theta} \d w\,\frac{1}{w-z} \, e^{- Nx(z-q) +Ny(w-q)}\left(\frac{z}{w}\right)^{ N}\prod_{j=1}^n\left(\frac{w-\lambda_j^{-1}}{z-\lambda_j^{-1}}\right),
\qe
 where the $q\in\R$ is a free parameter and we recall that the $\lambda_j$'s are the eigenvalues of $\bv\Sigma_N$.    
  $\Gamma$ and $\Theta$ are disjoint closed contours such that $\Gamma$ encloses all the $\lambda_j^{-1}$'s  whereas $\Theta$ encloses the origin, see for instance Figure \ref{fig:contour-KN}.  
  
  \paragraph{Convention: } All the contours we consider are simple and oriented counterclockwise. 
  
 \begin{figure}[h]
\centering
\includegraphics[width=0.6\linewidth]{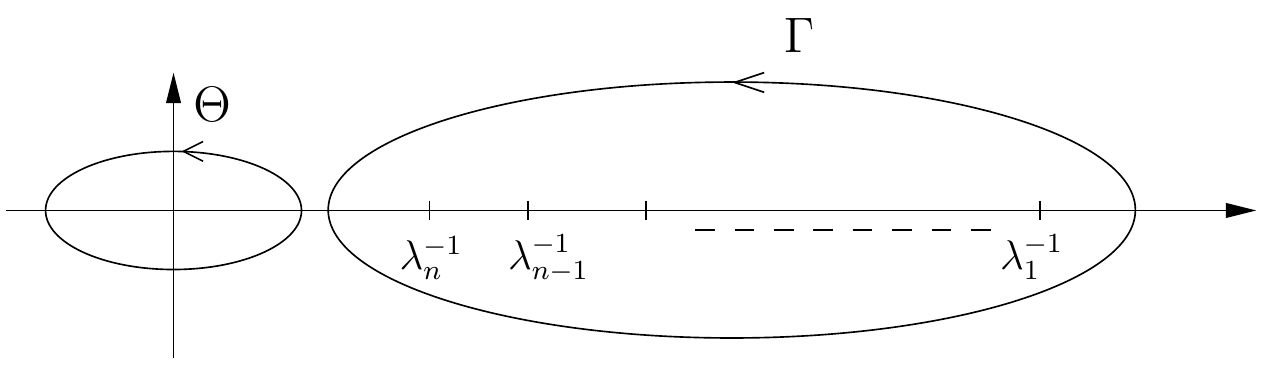}
\caption{The contours $\Gamma$ and $\Theta$ in the definition of  $\K_N(x,y)$}
\label{fig:contour-KN}
\end{figure} 

\begin{remark}
  In \cite{BBP-2005,On}, it is assumed that $q$ satisfies an extra restriction so that the associated operator $\K_N$ is trace class on the semi-infinite intervals $(s,+\infty)$.  Since we are here only interested in establishing a local uniform convergence for $\K_N(x,y)$, such restrictions are not necessary. See also \cite[Remark 4.3]{HHN-preprint}.
  \end{remark}

 In the remaining of this section, we assume that the framework (assumptions and notations) of Theorem \ref{th:main} holds true. We then set
\eq
\label{Ktilde}
\widetilde \K_N(x,y)=\frac{1}{ N^{3/4}\sigma_N} \K_N\left(\frak a_N+\frac{x}{N^{3/4}\sigma_N } , \frak a_N+\frac{y}{N^{3/4}\sigma_N }\right) ,
\qe
and, in order to establish  Theorem \ref{th:main}, focus on the proof of  the uniform convergence
\eq
\label{to show Th5}
\sup_{x,y\in[-s,s]}\Big|\widetilde\K_N(x,y)-\K^{(\tau)}_\Pe(x,y)\Big|\xrightarrow[N\to\infty]{} 0
\qe
for every fixed $s>0$.

\paragraph*{Notations and conventions:} \begin{itemize}
\item We denote by $B(z,\rho)$ the open disc in $\C$ with 
center $z\in \C$ and radius $\rho>0$.

\item By convention, we shall use at several instances a constant $C$ which depends neither on $N$ nor on $x,y\in [-s,s]$, but may depend on $s$, and whose exact value may change
from one ligne to another. 

\item If a contour $\Gamma$ is parametrized by $\gamma:I\rightarrow\Gamma$ for some interval $I\subset \R$, then  for every  map $h:\Gamma\rightarrow\C$ we set
\[
\int_{\Gamma}h(z)|\d z|=\int_Ih\circ\gamma(t)\,|\gamma'(t)|\d t
\]
when it does make sense. In particular, ${\oint}_\Gamma|\d z|$ is the length of a closed contour $\Gamma$.
\end{itemize}

\subsection{Preparation}

It follows from \eqref{KN} and \eqref{Ktilde}, by setting $q=\frak c_N$, that
\eq
\label{tKN int}
\widetilde \K_N(x,y)=\frac{N^{1/4}}{(2i\pi)^2\sigma_N}\oint_\Gamma\d z\oint_\Theta \d w\, \frac{1}{w-z}\,  e^{-\frac{N^{1/4}}{\sigma_N}x(z-\frak c_N)
+\frac{N^{1/4}}{\sigma_N}y(w-\frak c_N)} e^{Nf_N(z)-Nf_N(w)},
\qe
where we introduced the map
\eq
\label{fN}
f_N(z)=-\frak a_{N}(z-\frak  c_N)+\log(z)-\frac{1}{N}\sum_{j=1}^n\log(1-\lambda_j z)\ . 
\qe
Notice the functions $\exp(f_N)$, $\re f_N$,  and the derivatives $f_N^{(k)}$ ($k\ge 1$) are  well-defined on $\C\setminus\{0,\lambda_1^{-1},\ldots,\lambda_n^{-1}\}$. However, one needs to specify an appropriate determination of the complex logarithm in order to define $f_N$ properly.  By Assumption~\ref{ass:nu} and the regularity condition \eqref{RC}, there exists $\epsilon>0$ such that $\lambda_j^{-1}\in (0,+\infty)\setminus B(\frak c,\epsilon)$ for every $1\leq j\leq n$ and every $N$ large enough. As a consequence, if we introduce the compact set
\begin{equation}
\label{def:K}
{\mathcal K} = \left( \left[\inf_N \frac {1}{\lambda_n},\sup_N\frac{1}{\lambda_1}\right] \setminus B(\frak c,\epsilon)\right) \cup \{ 0\} \ ,
\end{equation}
then on every simply connected open subset of $\C\setminus\mathcal K$ one can find a determination of the logarithm so that $f_N$ is well-defined and holomorphic there.

Recalling the definition \eqref{gN} of $g_N$, an essential observation is that 
\eq
\label{fgN0}
f_N'(z)= g_N(z)-\frak a_N\ .
\qe 
As a consequence, we have for every $k\geq 1$,
\eq
\label{fgN1}
 f_N^{(k+1)}(z) =g_N^{(k)}(z)\ .
 \qe
In particular, since $\frak a_N=g_N(\frak c_N)$ by definition, the decay assumption \eqref{ass:speed} and  Proposition~\ref{prop:sequence-cN} provide
\begin{align}
 \label{fgN2}
 f'(\frak c_N)& =0\ ,\\
 \label{fgNspeed}
\sqrt N f_N''(\frak c_N)& \xrightarrow[N\to\infty]{} \tau\ ,\\
 \label{fgN3}
f^{(3)}(\frak c_N)& =0\ ,\\
\label{fgN4}
f_N^{(4)}(\frak c_N)& \xrightarrow[N\to\infty]{} g^{(3)}(\frak c)>0\ .
\end{align}

By performing a Taylor expansion of $f_N$ near $\frak c_N$ in \eqref{tKN int}, one can already guess from \eqref{fgN2}--\eqref{fgN4} and a change of variables that the Pearcey kernel should appear in the large $N$ limit, at least if one restricts the contours $\Gamma$ and $\Theta$ to a neighbourhood of $\frak c_N$ (see also \cite{HHN-esaim-preprint} for a more detailed heuristic, which may serve as a guideline for the forthcoming proof). In a first step, we provide precise estimates in order to prove that claim, see Section \ref{Pearcey step 1}. In a second step, we prove that the contribution coming from the pieces of contour away from $\frak c_N$ are exponentially negligible, see Section \ref{contour def}. To do so, we establish the existence of appropriate contours in the same fashion as in \cite{HHN-preprint}. This will enable us to conclude.  

Notice that one of the key arguments to assert the existence of appropriate contours is the maximum principle for subharmonic functions. This argument only appears in the proof of Lemma \ref{key level sets}, which is similar to the proof of \cite[Lemma 4.11]{HHN-preprint} and hence omitted. The interested reader may refer to \cite{HHN-preprint} for more details.

\subsection{Step 1: Local analysis around $\frak c_N$}\label{local analysis}
\label{Pearcey step 1}

We start with a quantitative Taylor expansion for $f_N$ near $\frak c_N$.

\begin{lemma}
\label{Taylor fN} 
There exists  $\rho_0>0$ and $\Delta=\Delta(\rho_0)>0$ independent of $N$ such
that for every $N$ large enough, $B(\frak c_N,\rho_0) \subset \C \setminus
{\mathcal K}$ and, whatever the analytic representation of $f_N$ on 
$B(\frak c_N,\rho_0)$, we have  for every $z\in B(\frak c_N,\rho_0)$
\[
\big|f_N(z)-f_N(\frak c_N)-\frac 12 g'_N(\frak c_N)(z-\frak c_N)^2-\frac{1}{4!}g^{(3)}_N(\frak c_N) (z-\frak c_N)^4\big|  \leq \Delta|z-\frak c_N|^5.
\]
In particular,  since $g_N'(\frak c_N)$ is real, for every 
$z\in B(\frak c_N,\rho_0)$,
\[
\big|\re f_N(z)-\re f_N(\frak c_N)-\frac 12 g'_N(\frak c_N)\re(z-\frak c_N)^2-\frac{1}{4! }g^{(3)}_N(\frak c_N) \re(z-\frak c_N)^4\big|
\leq \Delta|z-\frak c_N|^5.
\]
\end{lemma}

\begin{proof}
Choose $\rho_0$ such that $B(\frak c,2\rho_0)\subset \C\setminus\mathcal K$. The convergence $\frak c_N\to \frak c$ then yields  $B(\frak c_N,\rho_0)\subset B(\frak c,2\rho_0)\subset \C\setminus\mathcal
K$ for every $N$ large enough. By using \eqref{fgN1}--\eqref{fgN3} and performing a Taylor expansion for
$f_N$ around $\frak c_N$, we obtain

\begin{eqnarray*}
\lefteqn{\big|f_N(z)-f_N(\frak c_N)-\frac 12 g'_N(\frak c_N)(z-\frak c_N)^2-\frac{1}{4! }g^{(3)}_N(\frak c_N) (z-\frak c_N)^4\big| }\\
 &= & \quad \big|f_N(z)-f_N(\frak c_N)-\frac 12 f''_N(\frak c_N)(z-\frak c_N)^2-\frac{1}{4! }f^{(4)}_N(\frak c_N) (z-\frak c_N)^4\big|\\
&\leq  & \quad \frac{1}{5!}|z-\frak c_N|^5\max_{B(\frak c,2\rho_0)}|f_N^{(5)}|\ ,
\end{eqnarray*}
provided that $z\in B(\frak c_N,\rho_0)$ and $N$ is sufficiently large. 
Moreover, since  $f_N^{(5)}=g_N^{(4)}$ converges uniformly on
$B(\frak c,2\rho_0)$ to $g^{(4)}$ which is bounded there, the existence of
$\Delta=\Delta(\rho_0)$ independent of $N$ follows.
 \end{proof}
  
 From now, we let $\rho>0$ be small enough so that  $0<\rho<\rho_0$ and 
  \eq
 \label{cond rho}
 \frac{g_N^{(3)}(\frak c_N)}{4!}-\Delta \rho \geq 
\frac12\,\frac{g^{(3)}(\frak c)}{4!}
 \qe
%
We introduce the contours
\begin{align}
\label{s cont 1}
\Up_{o}& =\left\{\frak c_N+N^{-1/4}e^{-i\theta}:\quad \theta\in[-\pi/4,\pi/4]\right\}\ ,\\
\Up_{\times} & =\left\{\frak c_N-te^{ i\pi/4}:\quad  t\in[-\rho, -N^{-1/4}]\right\}\cup 
\left\{\frak c_N+te^{- i\pi/4}:\quad  t\in[N^{-1/4},\rho]\right\}
\ , \\
\Um_{o} & =\left\{\frak c_N-N^{-1/4}e^{-i\theta}:\quad \theta\in[-\pi/4,\pi/4]\right\}\ ,\\
\Um_{\times}& =\left\{\frak c_N+te^{i\pi/4}:\quad  t\in[-\rho,-N^{-1/4}]\right\}\cup \left\{\frak c_N+te^{3 i\pi/4}:\quad  t\in[N^{-1/4},\rho]\right\}\ .
\end{align}
The orientations of these contours are specified by letting the parameters $\theta$ and $t$ increase in their definition domains. We also introduce
\begin{align}
\Upsilon_{*}& =\Up_o\cup\Up_\times\cup\Um_o\cup \Um_\times\ ,\\
\label{s cont 2}
\Xi_* & =\left\{\frak c_N+te^{ i\pi/2}:\quad  t\in[-\rho,\rho]\right\}.
\end{align}
Similarly, the orientation of contour $\Xi_*$  is specified by letting $t$ increase in its definition domain.

We now establish the following estimate which essentially allows to replace $f_N$ by its Taylor expansion around $\frak c_N$ in the double integral over the contours $\Upsilon_{*}$ and $\Xi_*$.


\begin{lemma} 
\label{local step 1}
For every $s>0$, the following quantity 
\begin{align} 
\label{loc step formula}
& \frac{N^{1/4}}{(2i\pi)^2\sigma_N}\int_{\Upsilon_*}\d z\int_{\Xi_*} \d w\, \frac{1}{w-z}  \, e^{-N^{1/4}x\frac{(z-\frak c_N)}{\sigma_N}+N^{1/4}y\frac{(w-\frak c_N)}{\sigma_N}}\\
& \times \left\{e^{Nf_N(z)-Nf_N(w)}-e^{N \frac{g_N'(\frak c_N)}2 (z-\frak c_N)^2+N \frac{g_N^{(3)}(\frak c_N)}{4!}(z-\frak c_N)^4-N \frac{g_N'(\frak c_N)}2(w-\frak c_N)^2-N\frac{g_N^{(3)}(\frak c_N)}{4!} (w-\frak c_N)^4}\right\}\nonumber
\end{align}
converges to zero as $N\to \infty$, uniformly in $x,y\in[-s,s]$.
\end{lemma}

\begin{proof}
Let  $s>0$ be fixed, and set for convenience 
\begin{multline*}
D_N(z,w)=e^{Nf_N(z)-Nf_N(w)}\\
-e^{N \frac{g_N'(\frak c_N)}2 (z-\frak c_N)^2+N \frac{g_N^{(3)}(\frak c_N)}{4!}(z-\frak c_N)^4-N \frac{g_N'(\frak c_N)}2(w-\frak c_N)^2-N\frac{g_N^{(3)}(\frak c_N)}{4!} (w-\frak c_N)^4}\ ,
\end{multline*}
so that it amounts to prove that 
\begin{equation}\label{eq:intermediate-referee}
N^{1/4}\sup_{x,y\in[-s,s]}\Bigg|\int_{\Upsilon_*}\d z\int_{\Xi_*} \d w  \, \frac{1}{w-z} \, e^{-N^{1/4}x\frac{(z-\frak c_N)}{\sigma_N}+N^{1/4}y\frac{(w-\frak c_N)}{\sigma_N}}D_N(z,w)\Bigg|
\end{equation}
vanishes as $N\to\infty$.

First, by definition of the contours \eqref{s cont 1}--\eqref{s cont 2}, we have the bound 
$$
\frac{1}{|z-w|} \le  \sqrt{2}\, N^{1/4}
$$ for every $z\in \Upsilon_*$ and $w\in\Xi_*$. As a consequence,
\begin{align} 
\label{A0}
& N^{1/4}\Big|\int_{\Upsilon_*}\d z\int_{\Xi_*} \d w \, \frac{1}{w-z} \, e^{-N^{1/4}x\frac{(z-\frak c_N)}{\sigma_N}+N^{1/4}y\frac{(w-\frak c_N)}{\sigma_N}}D_N(z,w)\Big|\nonumber\\
\leq & \quad \sqrt{2N}\int_{\Upsilon_*}|\d z|\int_{\Xi_*} |\d w|\, e^{-N^{1/4}x\frac{\re(z-\frak c_N)}{\sigma_N}+N^{1/4}y\frac{\re(w-\frak c_N)}{\sigma_N}}\Big|D_N(z,w)\Big|\ .
\end{align}
Next, we use the following elementary inequality, 
\begin{eqnarray}
\label{ineq diff kernels}
 |e^u-e^v|  &=& e^{\re(v)}|e^{(u-v)}-1|\nonumber \\
 & \leq&    e^{\re(v)}\sum_{k\geq 1}\frac{|u-v|^k}{k!}
 \ \leq\   |u-v|e^{\re(v)+|u-v|}\ ,
\end{eqnarray}
which holds true for every $u,v\in\C$. By specializing it to 
\begin{eqnarray*}
u& =&Nf_N(z)-Nf_N(w)=N(f_N(z)-f_N(\frak c_N))-N(f_N(w)-f_N(\frak c_N))\ , \\
v& =&N \frac{g_N'(\frak c_N)}2 (z-\frak c_N)^2+N \frac{g_N^{(3)}(\frak c_N)}{4!} (z-\frak c_N)^4\\
&&  \qquad -N \frac{g_N'(\frak c_N)}2 (w-\frak c_N)^2-N\frac{g_N^{(3)}(\frak c_N)}{4!} (w-\frak c_N)^4 \ ,
\end{eqnarray*}
where $z\in \Upsilon_*$ and $w\in\Xi_*$, so that Lemma \ref{Taylor fN} yields $|u-v|\leq N\Delta(|z-\frak c_N|^5+|w-\frak c_N|^5)$ , we obtain 
\begin{eqnarray*}
\Big|D_N(z,w)\Big| & \leq&  N\Delta(|z-\frak c_N|^5+|w-\frak c_N|^5)\\
& & \quad \times\,  e^{N \frac{g_N'(\frak c_N)}2 \re(z-\frak c_N)^2+N \frac{g_N^{(3)}(\frak c_N)}{4!} \re(z-\frak c_N)^4 + N\Delta|z-\frak c_N|^5}\\
& & \qquad \times\,  e^{-N \frac{g_N'(\frak c_N)}2 \re(w-\frak c_N)^2-N \frac{g_N^{(3)}(\frak c_N)}{4!} \re(w-\frak c_N)^4 + N\Delta|w-\frak c_N|^5}.
\end{eqnarray*}
As a consequence, it follows
\begin{eqnarray}
\lefteqn{
 \Delta^{-1}\int_{\Upsilon_*}|\d z|\int_{\Xi_*} |\d w|\, e^{-N^{1/4}x\frac{\re(z-\frak c_N)}{\sigma_N}+N^{1/4}y\frac{\re(w-\frak c_N)}{\sigma_N}}\Big|D_N(z,w)\Big|}\nonumber \\
    &\leq  &   \int_{\Upsilon_*}N|z-\frak c_N|^5e^{-N^{1/4}x\frac{\re(z-\frak c_N)}{\sigma_N}+N\frac{g_N'(\frak c_N)}2 \re(z-\frak c_N)^2
    +N\frac{g_N^{(3)}(\frak c_N)}{4!}  \re(z-\frak c_N)^4+N\Delta|z-\frak c_N|^5}|\d z|\nonumber\\
 && \qquad \times \int_{\Ttilde_*}e^{N^{1/4}y\frac{\re(w-\frak c_N)}{\sigma_N}-N\frac{g_N'(\frak c_N)}2 \re(w-\frak c_N)^2-N\frac{g_N^{(3)}(\frak c_N)}{4!} 
 \re(w-\frak c_N)^4+N\Delta|w-\frak c_N|^5}|\d w|\nonumber\\
    && \quad  +  \int_{\Upsilon_*}e^{-N^{1/4}x\frac{\re(z-\frak c_N)}{\sigma_N}+N\frac{g_N'(\frak c_N)}2 \re(z-\frak c_N)^2+N\frac{g_N^{(3)}(\frak c_N)}{4!} 
    \re(z-\frak c_N)^4+N\Delta|z-\frak c_N|^5}|\d z|\nonumber\\
 && \qquad \times \int_{\Ttilde_*}N|w-\frak c_N|^5e^{N^{1/4}y\frac{\re(w-\frak c_N)}{\sigma_N}-N\frac{g_N'(\frak c_N)}2 \re(w-\frak c_N)^2-N\frac{g_N^{(3)}(\frak c_N)}{4!} \re(w-\frak c_N)^4+N\Delta|w-\frak c_N|^5}|\d w|\nonumber.\\
 \label{A1}
\end{eqnarray}

We now handle the integrals over each piece of contour separately.

First, consider the integrals over the contour $\Xi_*$. By definition $w\in\Xi_*$ if and only if there exists $t\in[-\rho,\rho]$ such that $w=\frak c_N+it$. Thus,
\[
\re (w-\frak c_N)=0,\qquad \re (w-\frak c_N)^2=-t^2,\qquad \re (w-\frak c_N)^4=t^4,\qquad |w-\frak c_N|^5=|t|^5.
\]
For $N$ large enough this yields, together with \eqref{cond rho},
%
\begin{eqnarray}
\label{A3}
\lefteqn{ \int_{\Ttilde_*}e^{N^{1/4}y\frac{\re(w-\frak c_N)}{\sigma_N}-N\frac{g_N'(\frak c_N)}2 \re(w-\frak c_N)^2-N\frac{g_N^{(3)}(\frak c_N)}{4!} \re(w-\frak c_N)^4
+N\Delta|w-\frak c_N|^5}|\d w|}\nonumber\\
& =&  \int_{-\rho}^\rho e^{N\frac{g_N'(\frak c_N)}2 t^2-N\frac{g_N^{(3)}(\frak c_N)}{4!} t^4+N\Delta |t|^5}\d t\nonumber\\
& \leq &\int_{-\rho}^\rho e^{N\frac{g_N'(\frak c_N)}2 t^2-Nt^4\left(\frac{g_N^{(3)}(\frak c_N)}{4!}-\rho\Delta\right)}\d t\nonumber\\
& \leq&\int_{-\rho}^\rho e^{N\frac{g_N'(\frak c_N)}2 t^2-N\xi t^4}\d t\nonumber\\
& = &\frac{1}{N^{1/4}}\int_{-\rho N^{1/4}}^{\rho N^{1/4}}e^{\sqrt N\frac{g_N'(\frak c_N)}2 t^2-\xi t^4}\d t\nonumber\\
& \leq &\frac{1}{N^{1/4}}\int_{-\infty}^{+\infty}e^{\sqrt N\frac{g_N'(\frak c_N)}2 t^2- \xi t^4}\d t\quad \leq\quad \frac{C}{N^{1/4}}
\end{eqnarray}
where we set for convenience 
$$
\xi \ =\ \frac 12\, \frac{g^{(3)}(\frak c)}{4!}\, >\ 0\, .
$$ 
For the last estimate, we have used the decay 
assumption $ \sqrt N\,g_N'(\frak c_N) \to \kappa\in\R$. Similarly,
\begin{align}
&\int_{\Ttilde_*}N|w-\frak c_N|^5e^{N^{1/4}y\frac{\re(w-\frak c_N)}{\sigma_N}-N\frac{g_N'(\frak c_N)}2 \re(w-\frak c_N)^2-N\frac{g_N^{(3)}(\frak c_N) }{4!}
\re(w-\frak c_N)^4+N\Delta|w-\frak c_N|^5}|\d w|\nonumber\\
& \quad = \quad  N\int_{-\rho}^\rho |t|^5 e^{N\frac{g_N'(\frak c_N)}2 t^2-N\frac{g_N^{(3)}(\frak c_N)}{4!} t^4+N\Delta | t|^5}\d t\nonumber\\
& \quad \leq \quad  N \int_{-\rho}^\rho |t|^5e^{N\frac{g_N'(\frak c_N)}2 t^2-N\xi t^4}\d t\nonumber\\
& \quad = \quad \frac{1}{\sqrt N}\int_{-\rho N^{1/4}}^{\rho N^{1/4}}|t|^5e^{\sqrt N\frac{g_N'(\frak c_N)}2 t^2-\xi t^4}\d t \quad \leq\quad \frac{C}{\sqrt N}\ .
\end{align}

Next, we turn to the integrals over the contours $\Up_{\times}$ and $\Um_{\times}$. By definition $z\in\Up_{\times}$, resp. $z\in\Um_{\times}$, if and only if there exists $t\in[N^{-1/4},\rho]$ such that $z=\frak c_N+e^{\pm i\pi/4}t$, resp. $z=\frak c_N-e^{\pm i\pi/4}t$, so that in both cases we have
\[
|\re (z-\frak c_N)|\leq t,\qquad \re (z-\frak c_N)^2=0,\qquad 
\re (z-\frak c_N)^4=-t^4,\qquad |z-\frak c_N|^5= t^5.
\]
As a consequence, for every $x\in[-s,s]$,
\begin{align}
\label{A4}
& \int_{\Up_{\times}\cup \Um_{\times}}e^{-N^{1/4}x\frac{\re(z-\frak c_N)}{\sigma_N}+N\frac{g_N'(\frak c_N)}2 \re(z-\frak c_N)^2+N\frac{g_N^{(3)}(\frak c_N)}{4!} 
\re(z-\frak c_N)^4+N\Delta|z-\frak c_N|^5}|\d z| \nonumber\\
& \quad \leq \quad 4\int_{N^{-1/4}}^\rho e^{N^{1/4}\frac{st}{\sigma_N}-N\frac{g_N^{(3)}(\frak c_N)}{4!} t^4+N\Delta |t|^5}\d t\nonumber\\
& \quad \leq \quad 4\int_{N^{-1/4}}^\rho e^{N^{1/4}\frac{st}{\sigma_N}-N\xi t^4}\d t\nonumber\\
& \quad = \quad\frac{4}{N^{1/4}}\int_{1}^{\rho N^{1/4}}e^{\frac{st}{\sigma_N}-\xi t^4}\d t \quad \leq\quad \frac{C}{N^{1/4}} 
\end{align}
where the last inequality follows from the fact that $\liminf_N \sigma_N > 0$. Similarly,
\begin{align}
& \int_{\Up_{\times}\cup \Um_{\times}}N|z-\frak c_N|^5e^{-N^{1/4}x\frac{\re(z-\frak c_N)}{\sigma_N}+N\frac{g_N'(\frak c_N)}2 \re(z-\frak c_N)^2+N\frac{g_N^{(3)}(\frak c_N)}{4!} \re(z-\frak c_N)^4+N\Delta|z-\frak c_N|^5}|\d z|\nonumber\\
& \quad \leq  \quad 4N\int_{N^{-1/4}}^\rho t^5 e^{N^{1/4}\frac{st}{\sigma_N}-N\frac{g_N^{(3)}(\frak c_N)}{4!} t^4+N\Delta |t|^5}\d t\nonumber\\
& \quad \leq \quad 4N \int_{N^{-1/4}}^\rho t^5e^{N^{1/4}\frac{st}{\sigma_N}-N\xi t^4}\d t\nonumber\\
& \quad = \quad\frac{4}{\sqrt N}\int_{1}^{\rho N^{1/4}}t^5e^{\frac{st}{\sigma_N}-\xi t^4}\d t\quad \leq\quad \frac{C}{\sqrt N}\ . 
\end{align}
Finally, we consider the integrals over the contours $\Up_o$ and $\Um_o$. By definition $z\in\Up_o\cup\Um_o$ if and only if there exists $\theta\in[-\pi/4,\pi/4]$ such that $z=\frak c_N\pm N^{-1/4}e^{ i\theta}$, and hence
\[
|\re (z-\frak c_N)|\leq N^{-1/4},\qquad |\re (z-\frak c_N)^2|\leq N^{-1/2},\qquad |\re (z-\frak c_N)^4|\leq N^{-1},\qquad |z-\frak c_N|^5\leq N^{-5/4}.
\]
As a consequence, for every $x\in[-s,s]$,
\begin{align}
\label{A5}
& \int_{\Up_o\cup \Um_o}e^{-N^{1/4}x\frac{\re(z-\frak c_N)}{\sigma_N}+N\frac{g_N'(\frak c_N)}2 \re(z-\frak c_N)^2+N\frac{g_N^{(3)}(\frak c_N)}{4!} \re(z-\frak c_N)^4+N\Delta|z-\frak c_N|^5}|\d z|\nonumber\\
& \quad \leq \quad  e^{\frac{s}{\sigma_N}+\sqrt N g_N'(\frak c_N)+g_N^{(3)}(\frak c_N)+\frac{\Delta}{N^{1/4}}} \int_{\Up_o\cup \Um_o}|\d z|\nonumber\\
& \quad = \quad \frac{\pi}{N^{1/4}} \, e^{\frac{s}{\sigma_N}+\sqrt N g_N'(\frak c_N)+g_N^{(3)}(\frak c_N)+\frac{\Delta}{N^{1/4}}} \quad \leq\quad \frac{C}{N^{1/4}}\ ,
\end{align}
where for the last estimate, we used the fact that  $\liminf_N \sigma_N > 0$, 
the decay assumption $\sqrt N\, g_N'(\frak c_N)\to \kappa\in\R$ and the 
convergence $g_N^{(3)}(\frak c_N)\to g^{(3)}(\frak c)$. Similarly,
\begin{align}
\label{A2}
&  \int_{\Up_o\cup \Um_o}N|z-\frak c_N|^5e^{-N^{1/4}x\frac{\re(z-\frak c_N)}{\sigma_N}+N\frac{g_N'(\frak c_N)}2 \re(z-\frak c_N)^2
+N\frac{g_N^{(3)}(\frak c_N)}{4!} \re(z-\frak c_N)^4+N\Delta|z-\frak c_N|^5}|\d z| \nonumber\\
& \quad \leq  \quad \frac{\pi}{\sqrt N} \, e^{\frac{s}{\sigma_N}+\sqrt N g_N'(\frak c_N)+g_N^{(3)}(\frak c_N)+\frac{\Delta}{N^{1/4}}} \quad \leq\quad \frac{C}{\sqrt N}\ .
\end{align}
By gathering \eqref{A1}--\eqref{A2}, we thus obtain the estimate
\[
\Delta^{-1}\int_{\Upsilon_*}|\d z|\int_{\Xi_*} |\d w|\, e^{-N^{1/4}x\frac{\re(z-\frak c_N)}{\sigma_N}+N^{1/4}y\frac{\re(w-\frak c_N)}{\sigma_N}}\Big|D_N(z,w)\Big|\leq \frac C{N^{3/4}}.
\]
Combined together with \eqref{A0}, this finally yields
\[
 N^{1/4}\sup_{x,y\in [-s,s]}\Bigg|\int_{\Upsilon_*}\d z\int_{\Xi_*} \d w \, \frac{1}{w-z} \, e^{-N^{1/4}x\frac{(z-\frak c_N)}{\sigma_N}+N^{1/4}y\frac{(w-\frak c_N)}{\sigma_N}}D_N(z,w)\Bigg|\leq \frac C{N^{1/4}}\ .
\] 
Hence, \eqref{eq:intermediate-referee} is proved, which in turn implies \eqref{loc step formula}. Proof of Lemma \ref{local step 1} is therefore complete.
\end{proof}

The next estimate completes the previous lemma by showing one can replace  the constant $Ng'_N(\frak c_N)$ in  \eqref{loc step formula}  by 
$-\sqrt N\tau/\sigma_N^2$, and that the resulting kernel is the Pearcey kernel, up to a negligible correction term. 

\begin{lemma} 
\label{local step 2}
For every $s>0$, 
\begin{multline*} 
\frac{N^{1/4}}{(2i\pi)^2\sigma_N}\int_{\Upsilon_*}\d z\int_{\Xi_*} \d w \, \frac{1}{w-z}\, e^{-N^{1/4}x\frac{(z-\frak c_N)}{\sigma_N} +N^{1/4}y\frac{(w-\frak c_N)}{\sigma_N}}\\
\times e^{N \frac{g_N'(\frak c_N)}2 (z-\frak c_N)^2+N \frac{g_N^{(3)}(\frak c_N)}{4!} (z-\frak c_N)^4-N \frac{g_N'(\frak c_N)}2(w-\frak c_N)^2-N\frac{g_N^{(3)}(\frak c_N)}{4!}(w-\frak c_N)^4}
\end{multline*}
converges as $N\to\infty$ towards $\K_\Pe^{(\tau)}(x,y)$, uniformly in $x,y\in[-s,s]$.
\end{lemma}

\begin{proof} Let $s>0$ be fixed and recall the definition \eqref{Pearcey cont} of $\K_\Pe^{(\tau)}(x,y)$. We first show that, uniformly in $x,y\in[-s,s]$, as $N\to\infty$,
\begin{multline}
\label{P trunc}
\K_\Pe^{(\tau)}(x,y)=\int_{\Sigma\,\cap\,\left\{|z|\leq  \frac{ N^{1/4}\rho}{\sigma_N}\right\}} \d z \int_{-iN^{1/4}\rho/\sigma_N}^{\,iN^{1/4}\rho/\sigma_N}\d w \, \frac{1}{w-z} \, e^{-xz-\frac{\tau}2 z^2+\frac{z^4}4+yw+\frac{\tau}2 w^2-\frac{w^4}4} \\+\bigO{\frac 1{N^{1/4}}}\ .
\end{multline}
Indeed, keeping in mind that the sequence $(\sigma_N)$ is bounded, it follows from the definition~\eqref{Sigma} of $\Sigma$ together with  simple estimates 
that for every $x,y\in[-s,s]$,
\begin{multline*}
 \Bigg|\int_{\Sigma\,\cap\,\left\{|z|\geq  \frac{N^{1/4}\rho}{\sigma_N}\right\}} \d z \int_{-i\infty}^{\,i\infty}\d w \, \frac{1}{w-z} \, e^{-xz-\frac{\tau}2 z^2+\frac{z^4}4+yw+\frac{\tau}2 w^2-\frac{w^4}4}\Bigg|\\
\leq  \quad  \frac{\sqrt 2\,\sigma_N}{N^{1/4}\rho} \,4\int_{N^{1/4}\rho/\sigma_N}^{+\infty}e^{s|t|-\frac{t^4}4}\d t\int_{-\infty}^{+\infty}e^{-\frac{\tau}2 u^2-\frac{u^4}4}\d u \quad \leq \frac{C}{N^{1/4}}\ ,
\end{multline*}
and 
\begin{multline*}
 \Bigg|\int_{\Sigma} \d z \left(\int_{-i\infty}^{-iN^{1/4}\rho/\sigma_N} + \int_{iN^{1/4}\rho/\sigma_N}^{i\infty}\right)\d w \, \frac{1}{w-z} \, e^{-xz-\frac{\tau}2 z^2+\frac{z^4}4+yw+ \frac{\tau}2 w^2-\frac{w^4}4}\Bigg|\\
\leq  \quad  \frac{\sqrt 2\,\sigma_N}{N^{1/4}\rho}\left(4\int_{1}^{+\infty}e^{s|t|-\frac{t^4}4}\d t+\frac1\pi e^{s+\tau+\frac 14}\right) 2\int_{N^{1/4}\rho/\sigma_N}^{+\infty}e^{-\frac{\tau}2 u^2-\frac{u^4}4}\d u\quad
 \leq  \quad  \frac{C}{N^{1/4}}\ ,
\end{multline*}
from which \eqref{P trunc} follows.

As a consequence, by performing the changes of variables $z\mapsto N^{1/4}(z-\frak c_N)/\sigma_N$ and $w\mapsto N^{1/4}(w-\frak c_N)/\sigma_N$ in the right hand side of \eqref{P trunc} and using the definition \eqref{def:constants} of $\sigma_N$, we obtain
\begin{multline}
\label{KP 1}
\K_\Pe^{(\tau)}(x,y)=\frac{N^{1/4}}{(2i\pi)^2\sigma_N}\int_{\Upsilon_*}\d z \int_{\Xi_*}\d w\, \frac{1}{w-z}\, e^{-N^{1/4}x\frac{(z-\frak c_N)}{\sigma_N}+N^{1/4}y\frac{(w-\frak c_N)}{\sigma_N}}\\
\times \Big(e^{-\sqrt N \tau \frac{(z-\frak c_N)^2}{2 \sigma_N^2} +N \frac{g_N^{(3)}(\frak c_N)}{4!}(z-\frak c_N)^4+ \sqrt N\tau  \frac{(w-\frak c_N)^2}{2\sigma_N^2}-N\frac{g_N^{(3)}(\frak c_N)}{4!}(w-\frak c_N)^4}\Big) +\bigO{\frac 1{N^{1/4}}}\ .
\end{multline}
In order to complete the proof of the proposition, we now prove the estimate \eqref{KP 1}  still holds true after replacement of the constant $-\sqrt N\tau/\sigma_N^2$  by $Ng'_N(\frak c_N)$. This amounts to showing that 
\begin{multline}
\label{KP 2} 
N^{1/4}\Bigg|\int_{\Upsilon_*}\d z\int_{\Xi_*} \d w\, \frac{1}{w-z} \, e^{-N^{1/4}x\frac{(z-\frak c_N)}{\sigma_N}+N \frac{g_N^{(3)}(\frak c_N)}{4!}(z-\frak c_N)^4
+N^{1/4}y\frac{(w-\frak c_N)}{\sigma_N}-N\frac{g_N^{(3)}(\frak c_N)}{4!} (w-\frak c_N)^4}\\
\times \left(e^{-\sqrt N \tau \frac{(z-\frak c_N)^2}{2\sigma_N^2} +\sqrt N \tau \frac{(w-\frak c_N)^2}{2\sigma_N^2} }-e^{  N g_N'(\frak c_N)\frac{(z-\frak c_N)^2}2- 
N g_N'(\frak c_N)\frac{(w-\frak c_N)^2}2}\right)\Bigg|
\end{multline}
converges to zero uniformly in $x,y\in[-s,s]$, which we now establish by using the same type of arguments as in the proof of Lemma~\ref{local step 1}.
To do so, we set 
\[
\Delta_N= -\frac\tau{2\sigma_N^2}- \sqrt N \frac{g_N'(\frak c_N)}2=\frac 12\left[\left(\frac{g_N^{(3)}(\frak c_N)}{g^{(3)}(\frak c)}\right)^{1/2} \kappa-\sqrt N\, g_N'(\frak c_N)\right]
\]
and observe that, since we have the convergences $g_N^{(3)}(\frak c_N)\to g^{(3)}(\frak c)$ and $\sqrt N\, g_N'(\frak c_N)\to\kappa$, $\Delta_N=o(1)$ as $N\to\infty$. First, since $|z-w|^{-1} \leq \sqrt{2}  N^{1/4}$ for every $z\in\Upsilon_*$ and $w\in\Xi_*$, \eqref{KP 2} is bounded from above by
\begin{multline}
\label{KP 3} 
\sqrt{2 N} \int_{\Upsilon_*}|\d z|\int_{\Xi_*} |\d w|\, e^{-N^{1/4}x\frac{\re(z-\frak c_N)}{\sigma_N}+N \frac{g_N^{(3)}(\frak c_N)}{4!} \re(z-\frak c_N)^4+N^{1/4}y\frac{\re(w-\frak c_N)}{\sigma_N}-N\frac{g_N^{(3)}(\frak c_N)}{4!}\re(w-\frak c_N)^4}\\
\times \Bigg|e^{-\sqrt N \tau \frac{(z-\frak c_N)^2}{2\sigma_N^2} +\sqrt N \tau \frac{(w-\frak c_N)^2}{2\sigma_N^2}}-e^{ N \frac{g_N'(\frak c_N)}2 (z-\frak c_N)^2
-N \frac{g_N'(\frak c_N)}2(w-\frak c_N)^2}\Bigg|.
\end{multline}
Next, we use inequality \eqref{ineq diff kernels} with 
\begin{eqnarray*}
u& =&- \sqrt N \tau \frac{(z-\frak c_N)^2}{2\sigma_N^2}+ \sqrt N \tau \frac{(w-\frak c_N)^2}{2\sigma_N^2}\ ,\\
 v  & =& N \frac{g_N'(\frak c_N)}2 (z-\frak c_N)^2-N \frac{g_N'(\frak c_N)}2 (w-\frak c_N)^2 \ ,
\end{eqnarray*}
so that $|u-v|\leq \sqrt N\Delta_N (|z-\frak c_N|^2+|w-\frak c_N|^2)$, in order to obtain  that \eqref{KP 3} is bounded from above by

\begin{align}
\label{}
 & \;  \int_{\Upsilon_*}  N \Delta_N|z-\frak c_N|^2e^{-N^{1/4}x\frac{\re(z-\frak c_N)}{\sigma_N}+N\frac{g_N'(\frak c_N)}2 \re(z-\frak c_N)^2 
 + \sqrt N\Delta_N|z-\frak c_N|^2+N\frac{g_N^{(3)}(\frak c_N) }{4!} \re(z-\frak c_N)^4}|\d z|\nonumber\\
 & \quad \times \int_{\Ttilde_*}e^{N^{1/4}y\frac{\re(w-\frak c_N)}{\sigma_N}-N\frac{g_N'(\frak c_N)}2 \re(w-\frak c_N)^2+\sqrt N\Delta_N|w-\frak c_N|^2-N\frac{g_N^{(3)}(\frak c_N)}{4!} \re(w-\frak c_N)^4}|\d w|\nonumber\\
    & \; +\,  \int_{\Upsilon_*}e^{-N^{1/4}x\frac{\re(z-\frak c_N)}{\sigma_N}+N\frac{g_N'(\frak c_N)}2 \re(z-\frak c_N)^2+\sqrt N\Delta_N|z-\frak c_N|^2+N\frac{g_N^{(3)}(\frak c_N)}{4!}  \re(z-\frak c_N)^4}|\d z|\nonumber\\
 & \quad \times \int_{\Ttilde_*}N\Delta_N|w-\frak c_N|^2e^{N^{1/4}y\frac{\re(w-\frak c_N)}{\sigma_N}-N\frac{g_N'(\frak c_N)}2 \re(w-\frak c_N)^2
 +\sqrt N\Delta_N|w-\frak c_N|^2-N\frac{g_N^{(3)}(\frak c_N)}{4!} \re(w-\frak c_N)^4}|\d w|\nonumber.
\end{align}
By performing essentially the same estimates as in \eqref{A3}--\eqref{A2}, we then prove that \eqref{KP 3} is a $\bigO{\Delta_N}$ as $N\to\infty$ uniformly in $x,y\in[-s,s]$. This completes the proof of Lemma~\ref{local step 2}.
\end{proof}

We have now completed the local analysis around $\frak c_N$. More precisely, by gathering Lemmas \ref{local step 1} and \ref{local step 2}, we have established the following result. 

\begin{proposition} 
\label{concl Section 1}
For every $s>0$, 
\[
\frac{N^{1/4}}{(2i\pi)^2\sigma_N}\int_{\Upsilon_*}\d z\int_{\Xi_*} \d w  \, \frac{1}{w-z} \, e^{-N^{1/4}x\frac{(z-\frak c_N)}{\sigma_N}+N^{1/4}y\frac{(w-\frak c_N)}{\sigma_N}+Nf_N(z)-Nf_N(w)}
\]
converges towards $\K_\Pe^{(\tau)}(x,y)$ uniformly in $x,y\in[-s,s]$ as $N\to\infty$.

\end{proposition}

\subsection{Step 2: Existence of appropriate contours}
\label{contour def}

In this section, we prove the existence of appropriate completions for the contours $\Um_o\cup\Um_\times$, $\Up_o\cup\Up_\times$ and $\Xi_*$ into closed contours, on which the contribution coming from $e^{N(f_N(z)-f_N(w))}$ in the kernel $\widetilde\K_N(x,y)$ will bring an exponential decay as $N$ increases to infinity.
More precisely, we establish the following.

\begin{proposition}
\label{contours prop}
For every $\rho>0$ small enough, there exist $N$-dependent contours $\Um$, $\Up$ and $\Ttilde$ which satisfy 
for every $N$ large enough the following properties. 
\begin{itemize}
\item[{\rm (1)}] 
\begin{itemize}
\item[{\rm (a)}]
$\Um$ encircles the $\lambda_j^{-1}$'s smaller than $\frak c_N$
\item [{\rm (b)}]
$\Up$ encircles all the $\lambda_j^{-1}$'s larger than $\frak c_N$
\item[{\rm (c)}]
$\Ttilde$ encircles all  the $\lambda_j^{-1}$'s smaller than $\frak c_N$ and the origin
\end{itemize}

\item[{\rm (2)}]
\begin{itemize}
\item[{\rm (a)}] $\Um=\Um_o\cup \Um_{\times}\cup \Um_{res}$
\item[{\rm (b)}] $\Up=\Up_o\cup\Up_{\times}\cup \Up_{res}$ 
\item[{\rm (c)}] $\Ttilde =\Ttilde_*\cup \Ttilde_{res}$ 
\end{itemize}
\item[{\rm (3)}] 
There exists $K>0$ independent of $N$ such that 
\begin{itemize}
\item[{\rm (a)}] $\re \big( f_N(z)-f_N(\frak c_N)\big)\leq -K$ for all $z\in\Um_{res}\cup \Up_{res}$
\item[{\rm (b)}] $\re \big( f_N(w)-f_N(\frak c_N)\big)\geq K$ for all $w\in\Ttilde_{res}$
\end{itemize}

\item[{\rm (4)}] 
There exists $d>0$ independent of $N$ such that  
\begin{align*}
&\inf\big\{|z-w|: \;z\in \Um_{res}\cup \Up_{res},\,w\in\Ttilde\big\}  \geq d\\
&\inf\big\{|z-w|: \;z\in \Um\cup \Up,\,w\in\Ttilde_{res}\big\} \geq d
\end{align*}

\item[{\rm (5)}]
\begin{itemize}
\item[{\rm (a)}]
The contours $\Um$, $\Up$ and $\Ttilde$ lie in a bounded subset of $\C$ independent of  $N$
\item[{\rm (b)}]
The lengths of $\Um$, $\Up$ and $\Ttilde$ are uniformly bounded in $N$.
\end{itemize}
\end{itemize}
\end{proposition}

In order to provide a proof for Proposition \ref{contours prop}, we use the same approach as in \cite[Section 4.4]{HHN-preprint}, from which we borrow a few lemmas.  Introduce the asymptotic counterpart of $\re f_N$, namely
\eq
\label{f}
\re f(z)=-\frak a\, \re (z- \frak c)+\log|z|-\gamma \int \log|1-xz|\, \nu(\d x),\qquad z\in\C\setminus\mathcal K,
\qe
where we recall that $\mathcal K$ has been introduced in~\eqref{def:K}. We shall use the following property.
  
\begin{lemma}  
\label{fN -> f}$\re f_N$ converges locally uniformly to $\re f$ on $\C\setminus\mathcal K$, and moreover,
\eq
\label{re conv}
\lim_{N\rightarrow\infty}\re f_N(\frak c_N)=\re f(\frak c).
\qe
\end{lemma}

\begin{proof}
See \cite[Lemma 4.7(a)]{HHN-preprint}.
\end{proof}

We now turn to a qualitative analysis for the map $\re f$.  To do so, introduce  the sets
\eq
\Omega_-=\big\{z \in\C:\; \re f(z)<\re f(\frak c)\big\},\qquad \Omega_+=\big\{z \in\C:\; \re f(z)>\re f(\frak c)\big\}.
\qe
The next lemma encodes the behaviours of $\re f$ as $|z|\to\infty$.
\begin{lemma}
\label{behaviour infinity}
Both $\Omega_+$ and $\Omega_-$ have a unique unbounded connected component. Moreover, given any $\alpha\in(0,\pi/2)$, there exists $R>0$ large enough such that
\begin{align}
\label{omega R -}
& \Omega_-^R = \left\{z\in\C : \quad |z|>R,\quad -\frac{\pi}{2}+\alpha<\arg(z)<\frac{\pi}{2}-\alpha\right\}\subset \Omega_-\ ,\\
\label{omega R +}
& \Omega_+^R = \left\{z\in\C : \quad |z|>R,\quad \frac{\pi}{2}+\alpha<\arg(z)<\frac{3\pi}{2}-\alpha\right\}\subset \Omega_+\ .
\end{align}
\end{lemma}

\begin{proof}
See \cite[Lemma 4.8]{HHN-preprint}.
\end{proof}
Next, we describe the behaviour of $\re f$ in a neighbourhood of $\frak c$. 

 \begin{lemma}
\label{triangle f}   
There exist $\eta>0$ and $\theta>0$ small enough such that, if we set 
\eq
\label{Delta k}
\Delta_k=\Big\{z\in \C:\quad 0<|z-\frak c|<\eta,\quad \left|\arg(z-\frak c)-k\frac{\pi}{4}\right|<\theta\Big\},
\qe
then 
\[
\Delta_{ \pm 1},\;\Delta_{\pm 3}\subset \Omega_-,\qquad \Delta_{0},\; \Delta_{\pm 2},\;\Delta_4\subset \Omega_+.
\]
\end{lemma}
The regions $\Delta_k$ are shown on Figure~\ref{fig:ptselle}.

\begin{proof} 
Let $\eta>0$ be small enough so that $ B(\frak c,2\eta)\subset\C\setminus\mathcal K$. In particular, one can choose a determination of the logarithm such that the map 
\[
 f(z)=-\frak a (z- \frak c)+\log(z)-\gamma \int \log(1-xz)\, \nu(\d x)
\]
is well-defined and holomorphic on $B(\frak c,2\eta)$, and its real part is
given by \eqref{f}.  Moreover, we have $f'(z)=g(z)-\frak a$. Since 
$\frak a=g(\frak c)$ and $g'(\frak c)=g''(\frak c)=0$, a Taylor expansion 
of $f$ around $\frak c$ then yields for every $z\in B(\frak c,\eta)$,
\begin{align*} 
\left| \re f(z)- \re f(\frak c) -
         \frac1{4!}g^{(3)}(\frak c)\re(z-\frak c)^4 \right| 
&\leq \left| f(z)- f(\frak c) -\frac1{4!}g^{(3)}(\frak c)(z-\frak c)^4\right|
\\ 
&\leq \frac{1}{5!}|z - \frak c|^5 \max_{B(\frak c,\eta)}|g^{(4)}| . 
\end{align*}  
Since $\re(z-\frak c)^4= (-1)^kr^4$ when $z=\frak c+re^{i k\pi /4}$, and 
because $g^{(3)}(\frak c)>0$, the lemma follows by choosing $\eta$ and then 
$\theta$ small enough.
\end{proof}

Let $\Omega_{2\ell +1}$ be the connected component of $\Omega_-$ which contains $\Delta_{2\ell+1}$. Similarly,  let $\Omega_{2\ell}$ be the connected component of $\Omega_+$ which contains $\Delta_{2\ell}$. We now prove that the following holds true.

\begin{lemma}\
\label{study level sets}
\begin{itemize}
\item[{\rm (1)}]
 We have $\Omega_{1}=\Omega_{-1}$, the interior of $\Omega_1$ is connected, and for every $0<\alpha<\pi/2$ there exists $R>0$ such that 
\eq
\label{omega 1 R}
\left\{z\in\C : \quad |z|>R,\quad -\frac{\pi}{2}+\alpha<\arg(z)<\frac{\pi}{2}-\alpha\right\}\subset \Omega_1\ .
\qe
\item[{\rm (2)}]
 We have $\Omega_{2}=\Omega_{-2}$, the interior of $\Omega_2$ is connected, and  there exists $R>0$ such that
\eq
\label{omega 2 R}
\left\{z\in\C : \quad |z|>R,\quad \frac{\pi}{2}+\alpha<\arg(z)<\frac{3\pi}{2}-\alpha\right\}\subset \Omega_2\ .
\qe
\item[{\rm (3)}]  We have $\Omega_{3}=\Omega_{-3}$, the interior of $\Omega_3$ is connected and there exists $\delta>0$ such that $B(0,\delta)\subset \Omega_3$.
\end{itemize}
\end{lemma}

To prove the lemma, we use the following key properties.

\begin{lemma}
\label{key level sets}
\begin{itemize} 
\item[{\rm (1)}]
If $\Omega_*$ is a  connected component of $\Omega_+$, then $\Omega_*$ is open and, if  $\Omega_*$ is moreover bounded,  there exists $x\in\supp(\nu)$ such that $x^{-1}\in\Omega_*$.
\item[{\rm (2)}]
Let  $\Omega_*$ be a connected component of $\Omega_-$ such that $\Omega_*\not\subset\R$. 
\begin{itemize}
\item[{\rm (a)}]
If $\Omega_*$ is bounded, then $0\in\Omega_*$. 
\item[{\rm (b)}]
If $\Omega_*$ is bounded, then its interior is connected.  
\item[{\rm (c)}]
If $0\notin\Omega_*$, then the interior of $\Omega_*$ is connected.
\end{itemize}
\end{itemize}
\end{lemma}

\begin{proof}
See \cite[Lemma 4.11]{HHN-preprint}.
\end{proof}

\begin{proof}[Proof of Lemma \ref{study level sets}] We first prove (2). Since
$\Omega_2$ is by definition a connected subset of $\Omega_+$, 
Lemma~\ref{key level sets}--(1) yields that its interior is connected 
(since $\Omega_2$ is
open). Next, we show by contradiction that $\Omega_2$ is unbounded. If
$\Omega_2$ is bounded,  then Lemma~\ref{key level sets}--(1) shows there exists
$x\in\supp(\nu)$ such that $x^{-1}\in\Omega_2$. If $x^{-1}<\frak c$ (resp.
$x^{-1}>\frak c$), then it follows from the symmetry 
$\re  f(\overline z)=\re f(z)$ that $\Omega_2$ completely surrounds 
$\Omega_3$ (resp. $\Omega_1$). As a consequence, $\Omega_3\not\subset\R$ 
(resp.~$\Omega_1\not\subset\R$) is a
bounded connected component of $\Omega_-$ which does not contain the origin,
and  Lemma~\ref{key level sets}--(2a) shows this is impossible. The symmetry
$\re  f(\overline z)=\re f(z)$ moreover provides that  $\Omega_{-2}$ is also
unbounded, and (2) follows from the inclusion~\eqref{omega R +} and the fact
that $\Omega_+$ has a unique unbounded connected component, see 
Lemma~\ref{behaviour infinity}.

We now turn to (1). Since $\Omega_2$ is unbounded, then $\Omega_1$ does not
contain the origin and it follows from Lemma~\ref{key level sets}(2a)-(2c)
that $\Omega_1$ is unbounded and has a connected interior. Then, (1) follows
from symmetry $\re  f(\overline z)=\re f(z)$, the inclusion~\eqref{omega R -}
and the fact that $\Omega_-$ has a unique unbounded connected component 
(cf. Lemma~\ref{behaviour infinity}).

Finally,  since $\Omega_3$ is bounded as a byproduct of 
Lemma~\ref{study level sets}--(2), it has a connected interior 
(Lemma~\ref{key level sets}--(2b)) and contains the origin 
(Lemma~\ref{key level sets}--(2a)). Moreover, since $\re f(z)\to -\infty$ as 
$z\to 0$, the origin belongs to its interior and, because of the symmetry 
$\re f(\overline z)=\re f(z)$, necessarily $\Omega_3=\Omega_{-3}$. Hence (3) is established. 
\end{proof}

\begin{proof}[Proof of Proposition \ref{contours prop}]
Given any $\rho>0$ small enough, it follows from the convergence of $\frak c_N$
to $\frak c$ that for every $N_0$ large enough the points 
$\frak c_{N_0}+\rho e^{ i\pi/4}$ and $\frak c_{N_0}+\rho e^{ -i\pi/4}$  belong to 
$\Delta_1$ and $\Delta_{-1}$ respectively. Thus both points belong to 
$\Omega_1$ by 
Lemma~\ref{study level sets}--(1). As a consequence, we can complete the path 
$\Up_o\cup\Up_\times$  
into a  closed contour with a
path $\Up_{res}(N_0)$ lying in the interior of $\Omega_1$. Since $\Up_{res}(N_0)$ lies in the interior of $\Omega_1$, the convergence 
$\frak c_N\rightarrow\frak c$
moreover yields that we can perform the same construction for all $N\geq N_0$
with $\Up_{res}(N)$ in a closed tubular neighbourhood $\mathcal T\subset \Omega_1$ of
$\Up_{res}(N_0)$.  By Lemma \ref{study level sets}--(1) again, we can moreover
choose $\Up_{res}(N_0)$ in a way that it has finite length and only
crosses the real axis at a real number lying on the right of $\mathcal K$. By
construction, this yields that the set $\mathcal T$ is compact and that the
$\Up_{res}(N)$'s can be chosen with a uniformly bounded length as long
as $N\geq N_0$.  Since $\Omega_1\subset \Omega_-$ there exists $K>0$ such that
$\re f(z)\leq \re f(\frak c)-3K$ on $\mathcal T$. Since moreover $\re f_N$ uniformly
converges to $\re f$ on $\mathcal T$  and $\re f_N(\frak c_N)\rightarrow \re f(\frak c
)$ according to Lemma \ref{fN -> f}, we can choose $N_0$
large enough such that $\re f_N\leq \re f+K$ on $\mathcal T$ and $\re f(\frak c)\leq \re
f_N(\frak c_N)+K$. This finally yields that $\re (f_N(z)-f_N(\frak c_N))\leq
-K$ for all $z\in \mathcal T$ and   proves the existence of a contour $\Up$ satisfying
the requirements of Proposition \ref{contours prop}, except for the point (4).

Similarly, the same conclusion for $\Xi$ follows from the same lines but
by using $\Omega_2$ instead of $\Omega_1$ and Lemma \ref{study level sets}--(2).

We now turn to the contour $\Um$. Given any $\rho>0$ small enough, for every $N_0$ large enough the points 
$\frak c_{N_0}-\rho e^{-i\pi/4}$ and $\frak c_{N_0}-\rho e^{i\pi/4}$ belong to 
$\Delta_3$ and $\Delta_{-3}$ respectively. Thus both points belong to 
$\Omega_3$ by 
Lemma~\ref{study level sets}--(3), which contains the origin in its interior. As a consequence, we can complete the path 
$\Um_o\cup\Um_\times$  
into a  closed contour with a
path $\Um_{res}(N_0)$ lying in the interior of $\Omega_3$, which encircles all the $\lambda_j^{-1}$'s smaller than $\frak c_N$ for every $N$ large enough, since we assumed $\liminf_{N\to\infty}\lambda_1>0$. Then, we can follow the same construction as we did for $\Up$ in order to prove the existence of a contour $\Um$ satisfying
the requirements of Proposition \ref{contours prop}, except for the point (4).

Finally, the item (4) of Proposition~\ref{contours prop} is clearly satisfied by construction  since the sets $\Omega_-$ and $\Omega_+$ are disjoint, and the proof of the proposition  is therefore complete.

\end{proof}

\begin{figure}[h]
\centering
\includegraphics[width=\linewidth]{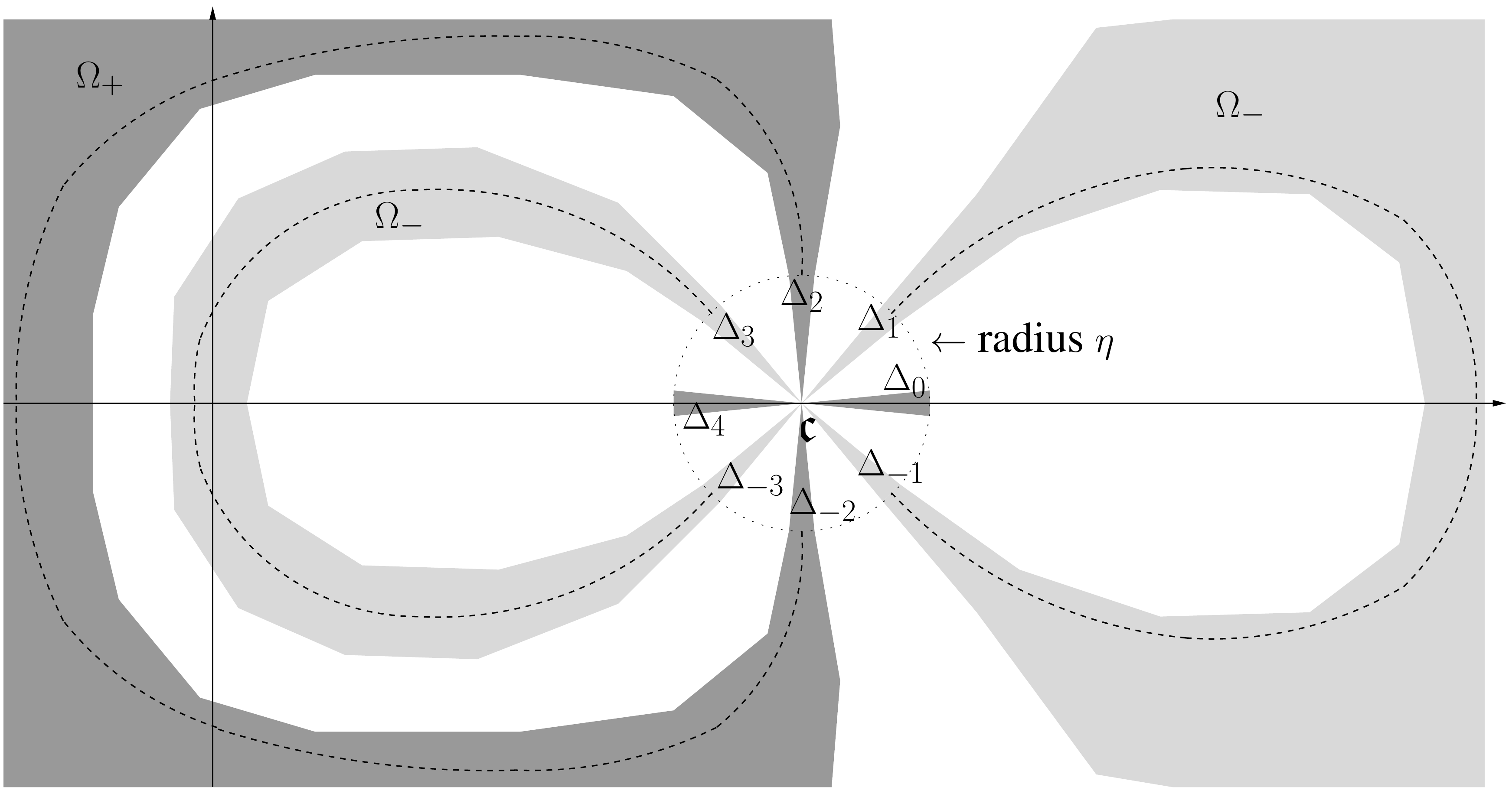}
\caption{Preparation of the saddle point analysis. The two dotted paths in
$\Omega_-$ correspond to $\Upsilon^+_{res}$ and $\Upsilon^-_{res}$ while the dotted path in $\Omega_+$ corresponds to $\Xi_{res}$. } 
\label{fig:ptselle} 
\end{figure}

Equipped with Proposition \ref{contours prop}, we are now in position to establish the remaining estimate towards the proof of Theorem \ref{th:main}. 

\begin{proposition} 
\label{concl Section 2}
For every $s>0$, the following quantity 
\[
N^{1/4}\left(\oint_{\Up\cup \Um}\d z\oint_{\Xi} \d w-\int_{\Upsilon_*}\d z\int_{\Xi_*} \d w\right) \frac{1}{w-z} \, e^{-N^{1/4}x\frac{(z-\frak c_N)}{\sigma_N}
+N^{1/4}y\frac{(w-\frak c_N)}{\sigma_N}+Nf_N(z)-Nf_N(w)}
\]
converges towards zero, uniformly in $x,y\in[-s,s]$, as $N\to\infty$.
\end{proposition}

\begin{proof}
Let $s>0$ be fixed. This amounts to showing that 
\eq
\label{int 1}
N^{1/4}\int_{\Up_{res}\cup\Um_{res}}\d z\oint_{\Xi} \d w \, \frac{1}{w-z} e^{-N^{1/4}x\frac{(z-\frak c_N)}{\sigma_N}+N^{1/4}y\frac{(w-\frak c_N)}{\sigma_N}
+Nf_N(z)-Nf_N(w)}
\qe
and 
\eq
\label{int 2}
N^{1/4}\oint_{\Upsilon}\d z\int_{\Xi_{res}} \d w \, \frac{1}{w-z} e^{-N^{1/4}x\frac{(z-\frak c_N)}{\sigma_N}+N^{1/4}y\frac{(w-\frak c_N)}{\sigma_N}+Nf_N(z)-Nf_N(w)}
\qe
converge to zero uniformly in $x,y\in[-s,s]$. First, Lemma \ref{Taylor fN} yields 
\[
\int_{\Upsilon_*}e^{N(\re f_N(w)-\re f_N(\frak c_N))}|\d z|\leq \int_{\Upsilon_*}e^{Ng_N'(\frak c_N)\re (z-\frak c_N)^2+Ng_N^{(3)}(\frak c_N)\re (z-\frak c_N)^4+N|z-\frak c_N|^5}\ ,
\]
and hence, by using the same estimates as in \eqref{A4} an \eqref{A5}, we obtain 
\eq
\int_{\Upsilon_*}e^{N(\re f_N(z)-\re f_N(\frak c_N))}|\d z|\leq \frac{C}{N^{1/4}}\ .
\qe
Similarly,  using instead \eqref{A3}, we get
\eq
\label{estim Xi fN}
\int_{\Xi_*}e^{-N(\re f_N(w)-\re f_N(\frak c_N))}|\d w|\leq \frac{C}{N^{1/4}}\ .
\qe
Next, Proposition \ref{contours prop} (5-b) yields the existence of $L>0$ independent of $N$ such that, for every  $z\in\Upsilon\cup\Xi$, we have $|\re (z-\frak c_N)|\leq L$. Since 
\[
e^{Nf_N(z)-Nf_N(w)}=e^{N(f_N(z)-f_N(\frak c_N))}e^{-N(f_N(w)-f_N(\frak c_N))},
\]
by using moreover Proposition \ref{contours prop} (4) and then Proposition \ref{contours prop} (3-a), we readily obtain
\begin{align*}
& N^{1/4}\Bigg|\int_{\Up_{res}\cup\Um_{res}}\d z\oint_{\Xi} \d w \, \frac{1}{w-z} e^{-N^{1/4}x\frac{(z-\frak c_N)}{\sigma_N}+N^{1/4}y\frac{(w-\frak c_N)}{\sigma_N}
+Nf_N(z)-Nf_N(w)}\Bigg|\\
\leq  & \qquad \frac{N^{1/4}}{d}\, e^{\frac{2N^{1/4}Ls}{\sigma_N}}\int_{\Um_{res}\cup\Up_{res}}e^{N(\re f_N(z)-\re f_N(\frak c_N))}|\d z| \oint_{\Xi}e^{-N(\re f_N(w)-\re f_N(\frak c_N))}|\d w|\\
\leq & \qquad \frac{N^{1/4}}{d} \, e^{\frac{2N^{1/4}Ls}{\sigma_N}-KN}\int_{\Um_{res}\cup\Up_{res}}|\d z| \left(\int_{\Xi_*}e^{-N(\re f_N(w)-\re f_N(\frak c_N))}|\d w|+e^{-NK}\int_{\Xi_{res}}|\d w|\right).
\end{align*}
Because of  Proposition \ref{contours prop} (5-b) and \eqref{estim Xi fN}, this yields the exponential decay of \eqref{int 1} to zero uniform in $x,y\in[-s,s]$. The same holds true for \eqref{int 2} by using the same arguments, up to the replacement of Proposition \ref{contours prop} (3-a) by (3-b), and hence the proof of Proposition \ref{concl Section 2} is complete.
\end{proof}

Finally, we can now easily conclude.

\subsection{Conclusion}
\begin{proof}[Proof of Theorem \ref{th:main}]
 Recalling the integral representation \eqref{tKN int} for  $\widetilde\K_N(x,y)$,
 we split the contour $\Gamma$ into two pieces $\Gamma^-$ and $\Gamma^+$, where $\Gamma^-$ encircles the $\lambda_j^{-1}$'s smaller  than $\frak c_N$, and $\Gamma^+$ the $\lambda_j^{-1}$'s larger  than $\frak c_N$. Then, we deform $\Theta$ so that it encircles $\Gamma^-$. This does not modify the value of the kernel. Indeed, there is no pole at $w=\lambda_j^{-1}$ and the residue picked at $w=z$ reads
$$
\frac{N^{1/4}}{2i\pi \sigma_N}e^{\frac{N^{1/4}}{\sigma_N}(y-x) (z-\frak c_N)}
$$  
and thus vanishes by Cauchy's theorem since the integrand is analytic.

Now, we can deform by analyticity $\Gamma^+$, $\Gamma^-$ and $\Theta$ into the respective contours  $\Upsilon^+$, $\Upsilon^-$ and $\Xi$ provided by Proposition \ref{contours prop}. Hence, $\widetilde \K_N(x,y)$ equals to
 \[
\frac{N^{1/4}}{(2i\pi)^2\sigma_N}\oint_{\Up\cup\Um}\d z\oint_\Xi \d w\, \frac{1}{w-z} \, e^{-N^{1/4}x\frac{(z-\frak c_N)}{\sigma_N}+N^{1/4}y\frac{(w-\frak c_N)}{\sigma_N}} e^{Nf_N(z)-Nf_N(w)},
 \]
 and the uniform convergence \eqref{to show Th5} follows for every $s>0$ from Proposition \ref{concl Section 1} and Proposition \ref{concl Section 2}. The proof of Theorem \ref{th:main} is therefore complete. 
\end{proof}

\section{Hard edge expansion: Proof of Theorem \ref{th:edge-expansion}}
\label{sec:proof-expansion}

In this section, we use basic properties of Fredholm determinants and trace class operators. We refer the reader to \cite{book-simon-2005, book-gohberg-2000} for comprehensive introductions, see also \cite[Section 4.2]{HHN-preprint} for a quick overview. We also use well-known formulas and identities for the Bessel function that can be found in \cite{book-erdelyi-1953}.

\subsection{Preparation and proof of Theorem \ref{th:edge-expansion}}
Recalling the definition \eqref{def:sigma-zeta} of $\sigma_N$ and the definition \eqref{KN} of the kernel $\K_N$, we consider here the scaled kernel
\begin{equation}\label{def:tilde-KN}
\widetilde \K_N(x,y) =\frac 1{N^2\sigma_N} \K_N\left( \frac x{N^2 \sigma_N}, \frac y{N^2 \sigma_N}\right) \ , \qquad x,y\in(0,s)\ , 
\end{equation}
and denote by $\widetilde \K_N$ the associated integral operator acting on $L^2(0,s)$. Thus, if $x_{\min}$ is the smallest random eigenvalue of $\bv M_N$ (see \eqref{x min}), we have 
\eq
\label{xmin det}
\mathbb{P}\Big( N^2 \sigma_N \, x_{\min} \ge s\Big) = \det \big( I-\widetilde \K_N\big)_{L^2(0,s)}\ .
\qe
Recalling the definition \eqref{Bessel kernel} of the Bessel kernel,  \cite[Proposition 6.1]{HHN-preprint} states that, for every $s>0$,  we have uniformly in $x,y\in (0,s)$ as $N\to\infty$,
\eq
\label{unif asymp hard}
\widetilde \K_N(x,y) = \left(\frac xy\right)^{\alpha/2}\K_\Be^{(\alpha)}(x,y) +\bigO{\frac1N}.
\qe
Notice that the limiting kernel appearing in the right hand side is exactly \eqref{Bessel kernel bis} without the pre-factor $(y/x)^{\alpha/2}$, and hence is bounded on $(0,s)\times (0,s)$. The integral operator associated  with this kernel is $\E \K_\Be^{(\alpha)}\E^{-1}$, where $\E$ acts on $L^2(0,s)$ by $\E f(x)=x^{\alpha/2}f(x)$. From \eqref{unif asymp hard}, it is easy to derive the convergence  
\eq
\label{fred conv bessel}
\det \big( I-\widetilde \K_N\big)_{L^2(0,s)}\xrightarrow[N\to\infty]{} \det \big( I- \K_\Be^{(\alpha)}\big)_{L^2(0,s)}=F_\alpha(x)\ .
\qe
Indeed,  \eqref{unif asymp hard} yields $\det ( I-\widetilde \K_N)_{L^2(0,s)}\to\det ( I- \E\K_\Be^{(\alpha)}\E^{-1})_{L^2(0,s)}$ and, for any trace class operators $\A$,$\B$ on $L^2(0,s)$ we have
\eq
\label{det switch}
\det(I-\A\B)_{L^2(0,s)}=\det(I-\B\A)_{L^2(0,s)}\ .
\qe
Thus, \eqref{fred conv bessel} follows because  both the operators $\E$ and $\K_{\Be}^{(\alpha)}\E^{-1}$ are well-defined on $L^2(0,s)$ and trace class when $\alpha\geq 0$, and so are the operators $\E\K_\Be^{(\alpha)}$ and $\E^{-1}$ when $\alpha<0$. See e.g. \cite[Section 6]{HHN-preprint} for a proof.

To prove Theorem  \ref{th:edge-expansion}, the first step is to improve \eqref{unif asymp hard} by making explicit the $1/N$-correction term, and showing the remainder is of order $1/N^2$. More precisely, with $\zeta_N$ defined as  in \eqref{def:sigma-zeta}, we prove the following kernel expansion. 

\begin{proposition} 
\label{prop: kernel exp}
For every $s>0$, we have uniformly in $x,y\in (0,s)$ as $N\to\infty$,
\begin{multline}
\label{unif expan hard}
\widetilde \K_N(x,y) \\
=  \left(\frac xy\right)^{\alpha/2}\left\{\K_\Be^{(\alpha)}(x,y) -  \frac{ \zeta_N}{4\sigma_N^2 N} \Big(\alpha J_\alpha(\sqrt x\,)J_\alpha(\sqrt y\, )+(x-y)\K_\Be^{(\alpha)}(x,y)\Big) \right\} +\bigO{\frac1{N^2}}.
\end{multline}
\end{proposition}
The proof of the proposition is deferred to Section \ref{sec:kernel exp}. If $\Q$ is the integral operator on $L^2(0,s)$ with kernel   
\eq
\label{upsilon kernel}
\Q(x,y) =  \frac{ \zeta_N}{4\sigma_N^2 } \Big(\alpha J_\alpha(\sqrt x\,)J_\alpha(\sqrt y\, )+(x-y)\K_\Be^{(\alpha)}(x,y)\Big),
\qe
then Proposition \ref{prop: kernel exp} yields the operator expansion (for the trace class norm on $L^2(0,s)$), 
\eq
\label{operator exp}
\widetilde \K_N = \E \Big( \K_\Be^{(\alpha)} - \frac1N \Q \Big)\E^{-1} +\bigO{\frac1{N^2}}.
\qe
Now, to prove Theorem \ref{th:edge-expansion}, one just has to plug \eqref{operator exp} into the Fredholm determinant \eqref{xmin det}, expand it, and identify the $1/N$ order term. For the last step, we need the following lemma, which essentially relies on a formula established by Tracy and Widom. 
 
\begin{lemma} 
\label{TW formula Bessel}
We have the identity
\[
\Tr \big((I-\K_\Be^{(\alpha)})^{-1} \Q\big)   =   -  \left(\frac{ \alpha\zeta_N}{\sigma_N^2 }\right) s \frac{\d}{\d s}\log\det\big(I-\Kbe\big)_{L^2(0,s)}\ .
\]
\end{lemma}

\begin{proof} Introduce the trace class operator $\mathrm{M}f(x)=x f(x)$ on $L^2(0,s)$, so that one can write 
\[
\Q = \frac{ \zeta_N}{4\sigma_N^2 } \Big(\alpha J_\alpha(\sqrt \cdot\,)\otimes J_\alpha(\sqrt \cdot\, )+[\mathrm{M},\Kbe]\Big)\ ,
\]
where $[\A,\B]=\A\B-\B\A$.
If  $\mathrm R(x,y)$ is the kernel of the resolvent $\Kbe(I-\Kbe)^{-1}$, then we have 
\[
\frac{\d}{\d s}\log\det\big(I-\Kbe\big)_{L^2(0,s)} = -\mathrm R(s,s).
\]
Equations (2.5) and (2.21) of \cite{tracy-widom-bessel-94} then provide the identity 
\eq
\label{TW bessel}
4s\mathrm R(s,s)=  \big\langle  J_\alpha(\sqrt \cdot\,), (I-\K_\Be^{(\alpha)})^{-1} J_\alpha(\sqrt \cdot\,)\big\rangle_{L^2(0,s)}\ .
\qe
Since the right hand side of \eqref{TW bessel} equals $\Tr \big((I-\K_\Be^{(\alpha)})^{-1} J_\alpha(\sqrt \cdot\,)\otimes  J_\alpha(\sqrt \cdot\,)\big)$, the lemma would follow provided that  
\[
\Tr \big((I-\K_\Be^{(\alpha)})^{-1} [\mathrm{M},\Kbe]\big) =0\ ,
\]
 but this is obvious since $\Tr(\A\B)=\Tr(\B\A)$ and $\Kbe$ commutes with $(I-\Kbe)^{-1}$.
\end{proof}

It is now easy to prove Theorem \ref{th:edge-expansion}.

\begin{proof}[Proof of Theorem \ref{th:edge-expansion}]
The identity \eqref{xmin det} and the operator expansion \eqref{operator exp} yield
\eq
\label{bessel A}
\mathbb{P}\Big( N^2 \sigma_N \, x_{\min} \ge s\Big)  =\det \big( I- \E\big( \K_\Be^{(\alpha)} -\frac1N \Q    \big)\E^{-1}\big)_{L^2(0,s)}  +\bigO{\frac1{N^2}}.
\qe
By plugging the identity $xJ'_{\alpha}(x) = \alpha J_{\alpha}(x) - x J_{\alpha+1}(x)$ into \eqref{Bessel kernel}, we obtain
\begin{equation}
\label{Bessel kernel ter}
\Kbe(x,y) = \frac{\sqrt{x} J_{\alpha+1}(\sqrt{x}\, ) J_{\alpha} (\sqrt{y}\,) - \sqrt{y} J_{\alpha+1} (\sqrt{y}\, ) J_{\alpha}(\sqrt{x}\,)}{2(x-y)}\ .
\end{equation}
Combined with the asymptotic behaviour (which follows from the definition \eqref{series rep Bessel}),
\[
J_\alpha(\sqrt x) =
\frac{(\sign(\alpha))^\alpha}{|\alpha| !}
\left(\frac{\sqrt x}{2}\right)^{|\alpha|} 
\big(1+O(x)\big)  \ , \qquad x\to 0_+\ ,
\]
it is then easy to check from \eqref{upsilon kernel} that both $\E \Q$ and $\Q\E^{-1}$ are trace class and thus, by \eqref{det switch},
\eq
\label{bessel B}
\det \big( I- \E\big( \K_\Be^{(\alpha)} -\frac1N \Q    \big)\E^{-1}\big)_{L^2(0,s)} =\det \big( I-  \K_\Be^{(\alpha)} +\frac1N \Q\big)_{L^2(0,s)} .
\qe
Finally, using the following expansion of a Fredholm determinant:
$$
\det( I-B) = 1-\Tr(B) + {\mathcal O} (\| B\|^2)\ ,
$$ 
where $B$ is trace class with trace norm $\|B\|<1$, we obtain
\begin{align}
 \det \big( I- \K_\Be^{(\alpha)} +\frac1N \Q \big)_{L^2(0,s)}
&  = \ \det \big( I- \K_\Be^{(\alpha)}\big)_{L^2(0,s)} \det\big(I+\frac1N( I-\K_\Be^{(\alpha)})^{-1} \Q \big)_{L^2(0,s)} \nonumber\\
& = \ \det \big( I- \K_\Be^{(\alpha)}\big)_{L^2(0,s)} \Big( 1 +\frac1N \Tr\big( (I-\K_\Be^{(\alpha)})^{-1} \Q\big) +\bigO{N^{-2}}\Big) \nonumber\\
\label{bessel C}
\end{align}
and Theorem \ref{th:edge-expansion} follows by combining \eqref{bessel A}, \eqref{bessel B}, \eqref{bessel C} together with Lemma \ref{TW formula Bessel}.
\end{proof}

We now turn to the proof of the proposition. 

\subsection{Proof of the kernel expansion}
\label{sec:kernel exp}
\begin{proof}[Proof of Proposition \ref{prop: kernel exp}] First, by using the representation \eqref{Bessel kernel ter} of the Bessel kernel and then  the identity $2\alpha J_{\alpha}(x) = x J_{\alpha+1}(x) + x J_{\alpha-1}(x)$, we have
\begin{align*}
& \alpha J_\alpha(\sqrt x\,)J_\alpha(\sqrt y\, )+(x-y)\K_\Be^{(\alpha)}(x,y) \\
&  = \;\frac12\big(2\alpha J_{\alpha}(\sqrt x\,)J_\alpha(\sqrt y\, ) + \sqrt x  J_{\alpha+1}(\sqrt x\,)J_\alpha(\sqrt y\, ) -  \sqrt y  J_\alpha(\sqrt x\,)J_{\alpha+1}(\sqrt y\, ) \big)\\
&  = \;\frac12\big(   \sqrt x  J_{\alpha+1}(\sqrt x\,)J_\alpha(\sqrt y\, ) +  \sqrt y  J_\alpha(\sqrt x\,)J_{\alpha-1}(\sqrt y\, ) \big).
\end{align*}
Thus, to prove Proposition \ref{prop: kernel exp} is equivalent to show
 \begin{multline}
\label{unif expan 2}
\widetilde \K_N(x,y) =  \left(\frac xy\right)^{\alpha/2} \K_\Be^{(\alpha)}(x,y) \\
-  \frac{ \zeta_N}{8\sigma_N^2 N}  \left(\frac xy\right)^{\alpha/2}\Big( \sqrt x  J_{\alpha+1}(\sqrt x\,)J_\alpha(\sqrt y\, ) +  \sqrt y  J_\alpha(\sqrt x\,)J_{\alpha-1}(\sqrt y\, ) \Big) +\bigO{\frac1{N^2}}
\end{multline}
uniformly in $x,y\in(0,s)$ as $N\to\infty$. To do so, let $0<r<R<\liminf_N\lambda_1/2$  and introduce the map 
$$
G_N(z)\ =\  \frac 1N \sum_{j=1}^n \log\left( \frac z{N\sigma_N} - \lambda_j\right).
$$
Recall that Assumption \ref{ass:nu} yields $\liminf_N \lambda_1>0$ and thus one can choose  
a determination of the logarithm so that $G_N$ is well-defined and holomorphic on 
$\{ z\in \C\, , \ |z|<R+1 \}$ for every $N$ large enough. In the proof of  \cite[Proposition 6.1]{HHN-preprint}, the following representation has been obtained
\eq
\label{KN bessel int}
\widetilde \K_N(x,y)= \frac 1{(2i\pi)^2} \oint_{|z|=r} \frac {\d z}z \oint_{|w|=R} \frac {\d w}w \frac 1{z-w} \left( \frac zw\right)^\alpha e^{-\frac xz +\frac yw -N\left( G_N(z) - G_N(w)\right)}.
\qe
Recalling the definitions \eqref{def:sigma-zeta} of $\sigma_N$ and $\zeta_N$, straightforward computations yield
\begin{eqnarray*}
G'_N(z) &=& \frac{1}{N^2\sigma_N} \sum_{j=1}^n \left( \frac z{N\sigma_N} - \lambda_j\right)^{-1}\ ,\quad G'_N(0)=-\frac 1{4N}\\
G''_N(z) &=& - \frac 1{N^3 \sigma_N^2} \sum_{j=1}^n \left(\frac z{N\sigma_N} - \lambda_j\right)^{-2}\ ,\quad G''_N(0)= - \frac {\zeta_N}{8N^2\sigma_N^2} \\
G_N^{(3)}(z) &=& \frac 2{N^4 \sigma_N^3} \sum_{j=1}^n \left(\frac z{N\sigma_N} - \lambda_j\right)^{-3}\ ,
\end{eqnarray*}
and hence we have the Taylor expansion
\eq
\label{Taylor hard}
G_N(z)= G_N(0) - \frac {z}{4N}- \frac {\zeta_N}{16N^2\sigma_N^2}z^2 +\bigO{\frac 1{N^3}},
\qe 
 uniformly valid for $|z|\le R+1$. Plugging this Taylor expansion into \eqref{KN bessel int} and recalling the contour integral representation \eqref{Bessel kernel bis} of the Bessel kernel, we readily obtain as $N\to\infty$,
\begin{eqnarray}
\label{Taylor kernel hard}
\widetilde \K_N(x,y)&=& \frac 1{(2i\pi)^2} \oint_{|z|=r} \frac {\d z}z \oint_{|w|=R} \frac {\d w}w \frac 1{z-w} \left( \frac zw\right)^\alpha 
e^{-\frac xz +\frac yw +\frac z4 -\frac w4}\nonumber\\
&&\qquad
 \times \exp\left(\frac {\zeta_N}{16N\sigma_N^2}  (z^2-w^2)+{\mathcal O}\left( \frac 1{N^2}\right)\right)\nonumber\\
 &=& \frac 1{(2i\pi)^2} \oint_{|z|=r} \frac {\d z}z \oint_{|w|=R} \frac {\d w}w \frac 1{z-w} \left( \frac zw\right)^\alpha 
e^{-\frac xz +\frac yw +\frac z4 -\frac w4}\nonumber\\
&&\qquad \times \left( 1+\frac{\zeta_N}{16N\sigma_N^2} (z^2-w^2) + {\mathcal O}\left( \frac 1{N^2}\right) \right)\nonumber\\
 &=& \left(\frac xy\right)^{\alpha/2} \Kbe(x,y) \nonumber\\
 &&\qquad + \left(\frac{\zeta_N}{16N\sigma_N^2}\right) \frac 1{(2i\pi)^2} \oint_{|z|=r} \frac {\d z}z \oint_{|w|=R} \frac {\d w}w \, (z+w) \left( \frac zw\right)^\alpha 
e^{-\frac xz +\frac yw +\frac z4 -\frac w4} \nonumber\\
&& \qquad \qquad +\; \bigO{\frac1{N^2}}
\end{eqnarray}
uniformly for $x,y\in (0,s)$. In the light of \eqref{unif expan 2}, we are left to show that
\begin{multline}
 \frac 1{(2i\pi)^2} \oint_{|z|=r} \frac {\d z}z \oint_{|w|=R} \frac {\d w}w \, (z+w) \left( \frac zw\right)^\alpha 
e^{-\frac xz +\frac yw +\frac z4 -\frac w4}\\
= -   2\left(\frac xy\right)^{\alpha/2}\Big( \sqrt x  J_{\alpha+1}(\sqrt x\,)J_\alpha(\sqrt y\, ) + \sqrt y  J_\alpha(\sqrt x\,)J_{\alpha-1}(\sqrt y\, ) \Big)
\end{multline}
in order to complete the proof of the proposition. But this easily follows from the contour integral representations of the Bessel function, 
\eq
\label{int rep Bessel}
J_{\alpha}(\sqrt{x}\,) = \frac{(-1)^{\alpha}}{2i\pi(2\sqrt{x}\,)^{\alpha}} \oint_{|z|=r} z^{\alpha} e^{-\frac xz +\frac z4}\frac {\d z}z\ ,\quad 
J_{\alpha}(\sqrt{y}\,) = (-1)^{\alpha} \frac{(2\sqrt{y}\,)^\alpha}{2i\pi} \oint_{|w|=R} \frac{e^{\frac yw - \frac w4}}{w^\alpha} \frac{\d w}{w}\ ,
\qe
see for instance \cite[Eq. (164) and (165)]{HHN-preprint} and perform the respective changes of variables $z\mapsto -z^{-1}$ and $w\mapsto -w^{-1}$. The proof of the proposition is therefore complete.

\end{proof}

\begin{remark} 
\label{next order exp}
By extending the Taylor expansion \eqref{Taylor hard} to higher order terms, the computation \eqref{Taylor kernel hard} easily yields the kernel expansion as $N\to\infty$, uniform in $x,y\in(0,s)$,
\[
\widetilde \K_N(x,y) = \left(\frac xy\right)^{\alpha/2} \Big\{\Kbe(x,y)  + \sum_{\ell=1}^L \frac{1}{N^\ell}\Q_N^{(\ell)}(x,y)\Big\} +\bigO{\frac{1}{N^{L+1}}}
\]
for every $L\geq 1$. The kernels $\Q_N^{(\ell)}(x,y) $ can be expressed in terms of sums of Bessel functions (by using \eqref{int rep Bessel}). By plugging this formula into the Fredholm determinant \eqref{xmin det}, and expanding it, this yields an asymptotic expansion of the form \eqref{hard edge exp ?}, although we are not able to provide a simple representation for the  coefficients $C_{N,\ell}^{(\alpha)}(s)$ when $\ell\geq 2$. It would be interesting to identify these coefficients in terms of $F_\alpha(s)$ if it is possible.
\end{remark}

\end{document}